\documentclass[10pt]{article}
\font\smallit=cmti10
\font\smalltt=cmtt10

\usepackage{amssymb,latexsym,amsmath,epsfig,amsthm}

\usepackage{graphicx}
\usepackage{amsfonts}
\usepackage{float}

\makeatletter

\renewcommand\section{\@startsection {section}{1}{\z@}
{-30pt \@plus -1ex \@minus -.2ex}
{2.3ex \@plus.2ex}
{\normalfont\normalsize\bfseries\boldmath}}

\renewcommand\subsection{\@startsection{subsection}{2}{\z@}
{-3.25ex\@plus -1ex \@minus -.2ex}
{1.5ex \@plus .2ex}
{\normalfont\normalsize\bfseries\boldmath}}

\renewcommand{\@seccntformat}[1]{\csname the#1\endcsname. }

\makeatother

\newtheorem{theorem}{Theorem}
\newtheorem{lemma}{Lemma}

\newtheorem{proposition}[lemma]{Proposition}

\theoremstyle{definition}
\newtheorem{definition}[lemma]{Definition}
\newtheorem{conjecture}[theorem]{Conjecture}
\newtheorem{remark}[lemma]{Remark}

\numberwithin{equation}{section}
\numberwithin{lemma}{section}

\newcommand{\NN}{\mathbb{N}}

\newcommand{\ZZ}{\mathbb{Z}}

\newcommand{\bfa}{\mathbf{a}}
\newcommand{\bfb}{\mathbf{b}}

\newcommand{\bfar}{\overleftarrow{\bfa}}

\newcommand{\bfap}{\bfa^+}

\newcommand{\mat}[1]{\begin{pmatrix} 0 & 1\\ 1 & #1\end{pmatrix}}

\newcommand{\Fundint}[1]{\left\langle #1\right\rangle}
\newcommand{\fundint}[1]{\langle #1\rangle}

\begin{document}

\begin{center}
\uppercase{\bf \boldmath 
An Elementary Characterization of the 
Gauss--Kuzmin Measure in the Theory of Continued Fractions}
\vskip 20pt
{\bf Shreyas Singh }\\
{\smallit 
Department of Mathematics,
California Institute of Technology, 
Pasadena, California}\\
{\tt ssingh3@caltech.edu}
\vskip 10pt
{\bf Zhuo Zhang }\\
{\smallit 
Department of Mathematics,
Stanford University,
Stanford, California}\\
{\tt zhuozh@stanford.edu}
\vskip 10pt
{\bf AJ Hildebrand}\\
{\smallit 
Department of Mathematics, 
University of Illinois, 
Urbana, Illinois}\\
{\tt ajh@illinois.edu}
\end{center}
\vskip 20pt

\centerline{\smallit Received: , Revised: , Accepted: , Published: }
\vskip 30pt

\centerline{\bf Abstract}
\noindent
By a classical result of Gauss and Kuzmin, 
the frequency with which a string
$\bfa=(a_1,\dots,a_n)$ of positive integers appears in the continued
fraction expansion of a random real number
is given by $\mu_{GK}({I(\bfa)})$,
where $I(\bfa)$ is the set of real numbers in $[0,1)$ whose
continued fraction expansion begins with the string $\bfa$ and
$\mu_{GK}$ is the \emph{Gauss--Kuzmin measure}, defined by $\mu_{GK}(I)=
\frac{1}{\log 2}\int_I \frac{1}{1+x} dx$, for any interval
$I\subseteq[0,1]$. 
It is known that the Gauss--Kuzmin measure satisfies the symmetry
property  $(*)$ $\mu_{GK}(I(\bfa))=\mu_{GK}(I(\bfar))$, where
$\bfar=(a_n,\dots,a_1)$ is the reverse of the string $\bfa$. We show 
that this property in fact characterizes the Gauss--Kuzmin measure: If
$\mu$ is any probability measure with continuous density function on
$[0,1]$ satisfying $\mu(I(\bfa))=\mu(I(\bfar))$ for  all finite strings
$\bfa$, then $\mu=\mu_{GK}$. 
We also consider the question whether symmetries analogous to $(*)$ hold
for permutations of $\bfa$ other than the reverse $\bfar$; we call such
a symmetry \emph{nontrivial}. We show that  
strings $\bfa$ of length $3$ have no nontrivial symmetries,
while for each $n\ge 4$ there exists an infinite
family of strings $\bfa$ of length $n$ that do have nontrivial
symmetries. Finally we present numerical data supporting the  
conjecture that, in an appropriate asymptotic sense, ``almost all''
strings $\bfa$  have no nontrivial symmetries.

\pagestyle{myheadings}
\markright{\smalltt INTEGERS: 25 (2025)\hfill}
\thispagestyle{empty}
\baselineskip=12.875pt
\vskip 30pt


\section{Introduction}
\label{sec:intro}

If one picks a random real number $x$ and expands it in base
$10$, then $1/10$ of the digits will be $0$, $1/10$ will be $1$, and so
on. More generally, any finite string $(d_1,\dots,d_n)$
of digits in $\{0,1,\dots,9\}$
occurs in the decimal expansion of the number with frequency $1/10^n$
in the sense that 
\begin{equation}
\label{eq:decimal-normality-def}
\lim_{N\to\infty}\frac1N\#\{0\le i\le N-1: 
d_{i+1}(x)=d_1,\dots,d_{i+n}(x)=d_n\} =\frac{1}{10^n},
\end{equation}
where $d_1(x),d_2(x),\dots$ is the sequence of decimal digits of the
fractional part of $x$.   A number $x$ with this property is called
\emph{normal} with respect to base $10$; normality with respect to other
bases is defined analogously.  By a classical result of 
Borel \cite{borel1909}, almost all real numbers $x$ are
normal with respect to all integer bases $n\ge2$.

In this paper we consider analogous questions for 
continued fraction expansions of numbers, that is, expansions of the form  
\begin{equation}
\label{eq:cf-def}
x=a_0(x)+
\cfrac{1}{a_1(x)+\cfrac{1}{a_2(x)+\cfrac{1}{\ddots}}} = [a_0(x);
a_1(x),a_2(x),\dots],
\end{equation}
where $a_0(x)=\lfloor x\rfloor$
and $a_i(x)$, $i=1,2,\dots$, are positive
integers, which we call the \emph{continued fraction digits}\footnote{%
In the literature, the numbers $a_i(x)$ are usually called \emph{partial
quotients}. We use the term \emph{digits} here to emphasize the
analogy to digits in ordinary decimal and base $b$ expansions.}
of $x$.  
The digits $a_i(x)$, $i=1,2,\dots$,  
can be computed from $x$ through a simple recursive procedure 
(see, for example,  \cite[Section 1.3]{dajani-kraaikamp-book}).
It is well-known (see Section \ref{sec:background} below for 
further details and references)
that any irrational number $x$ has a unique \emph{infinite}
continued fraction expansion of the form 
\eqref{eq:cf-def}, and that, conversely, for any integer $a_0$ and any 
infinite sequence $a_1,a_2,\dots$ 
of positive integers there exists a unique irrational
number $x$ whose continued fraction digits are given by this sequence. 
Thus, continued fraction expansions are analogous to decimal and base
$b$ expansions in that they provide a way to ``encode'' real numbers  in
terms of sequences of integers.

The continued fraction analogs of Borel's results on the frequencies of
digits and strings of digits in base $b$ expansions of real
numbers are the following theorems, which have their origins in work of
Gauss and which are key results in the metric theory of continued
fractions developed by Kuzmin \cite{kuzmin1929}.  

\begin{proposition}
[Gauss--Kuzmin Theorem; {\cite[(3.25)]{neverending-fractions-book}},
{\cite[Proposition 4.1.1]{iosifescu-kraaikamp-book}}]
Almost all real numbers $x$ satisfy
\begin{equation}
\notag
\lim_{N\to\infty}\frac1N\#\{1\le i\le N: 
a_{i}(x)=a\} = \log_2\left(1+\frac1{a(a+2)}\right)
\end{equation}
for every positive integer $a$, where 
$\log_2 x=(\log x)/(\log 2)$ denotes the base $2$ logarithm.
\end{proposition}

Thus, for almost all $x$, the continued fraction expansion of $x$ contains the
digit $1$ with frequency $\log_2(1+1/3)= 0.415037\dots$, 
the digit $2$ with frequency $\log_2(1+1/8)=0.169925\dots$, 
and so on.  The numbers $P_{GK}(a)$ defined by 
\begin{equation}
\label{eq:gk-distribution-def}
P_{GK}(a)=\log_2\left(1+\frac1{a(a+2)}\right)
\quad (a\in\NN)
\end{equation}
form a discrete probability distribution on $\NN$, called the 
\emph{Gauss--Kuzmin distribution}.


More generally, given a finite string $\bfa=(a_1,\dots,a_n)$ of
positive integers, set\footnote{There is a slight ambiguity in this
definition due to the ambiguity (see
\eqref{eq:cf-nonuniqueness} below) in the continued fraction representation
of a \emph{rational} number. This ambiguity does not affect the 
results stated here 
since rational numbers represent a set of Lebesgue measure $0$
and we are only concerned with integrals over the sets $I(\bfa)$,  
but it could be resolved by requiring the last digit in the continued
fraction representation of a rational number to be strictly greater than
$1$.}
\begin{equation}
\label{eq:Ibfa-def}
I(\bfa)=\{x\in[0,1): a_i(x)=a_i\quad (i=1,\dots,n)\}.
\end{equation}
Thus, $I(\bfa)$ is the set of real numbers in $[0,1)$ 
whose continued fraction expansion (ignoring the leading term $a_0(x)=0$) 
\emph{begins} with the string $\bfa$.
Let $\mu_{GK}$ be the 
\emph{Gauss--Kuzmin measure} on the interval $[0,1]$ defined by
\begin{equation}
\label{eq:mugk-def}
\mu_{GK}(I)=\frac1{\log 2}\int_I \frac{1}{1+x}\,dx,
\end{equation}
for any interval $I\subset [0,1]$. Finally, set
\begin{equation} 
\label{eq:Pbfa-formula}
P_{GK}(\bfa)=\mu_{GK}(I(\bfa))=\frac1{\log
2}\int_{I(\bfa)}\frac1{1+x}\,dx.
\end{equation}
With these notations we have the following result.

\begin{proposition}
[Generalized Gauss--Kuzmin Theorem; 
{\cite[Proposition 4.1.2]{iosifescu-kraaikamp-book}}]
Almost all real numbers $x$ satisfy 
\begin{equation}
\notag
\lim_{N\to\infty}\frac1N\#\{0\le i\le N-1: 
a_{i+1}(x)=a_1,\dots,a_{i+n}(x)=a_n\} =P_{GK}(\bfa)
\end{equation}
for every finite string $\bfa=(a_1,\dots,a_n)$ of positive integers.
\end{proposition}

Thus, $P_{GK}(\bfa)$ is the frequency with which a random real
number contains the string $\bfa$ in its continued fraction expansion.
When restricted to strings $\bfa$ of a \emph{fixed} length $n$, the 
frequencies $P_{GK}(\bfa)$ form a discrete probability measure on the
set $\NN^n$.

As an illustration of this result, consider the string $\bfa=(3,1,4)$. 
In this case the set $I(\bfa)$ is an interval with endpoints
\begin{equation}
\notag
[0;3,1,4] = \cfrac{1}{3+\cfrac{1}{1+\cfrac{1}{4}}} = \frac{5}{19},
\quad
[0;3,1,5] = \cfrac{1}{3+\cfrac{1}{1+\cfrac{1}{5}}} = \frac{6}{23}.
\end{equation}
Thus, by the generalized Gauss--Kuzmin Theorem the string $(3,1,4)$ occurs in the
continued fraction expansion of a random number $x$ with frequency 
\begin{align}
P_{GK}((3,1,4))&=
\mu_{GK}\left(\left(\frac6{23},\frac{5}{19}\right]\right)=
\frac{1}{\log 2}\int_{6/23}^{5/19}\frac{1}{1+x}\,dx
\notag
\\
\label{eq:example-P314}
&=\log_2\left(\frac{1+\frac{5}{19}}{1+\frac{6}{23}}\right)
= -\log_2\left(1-\frac1{551}\right)=0.002620\dots
\end{align}
(Recall that $\log_2(x)=(\log x)/\log 2)$ denotes the logarithm of $x$ with
respect to base $2$.)
For comparison, by \eqref{eq:decimal-normality-def}, the
frequency with which this string occurs in the \emph{decimal} expansion
of a random number is $1/10^{3}$.

For single digit strings $\bfa=(a)$, the set $I(\bfa)$ reduces to
the interval $(1/(a+1),1/a]$ and the frequency $P_{GK}(\bfa)$ becomes
\begin{align}
P_{GK}(\bfa)=\mu_{GK}\left(\left(\frac1{a+1},\frac1a\right]\right)
&=\frac1{\log
2}\int_{1/(a+1)}^{1/a} \frac{1}{1+x}dx
\notag
\\
=\log_2\left(1+\frac1{a(a+2)}\right),
\notag
\end{align} 
which is the Gauss--Kuzmin distribution $P_{GK}(a)$ defined in 
\eqref{eq:gk-distribution-def}.

Although continued fraction expansions share many properties 
with ordinary decimal and base $b$ expansions, there are three key 
differences. The most obvious difference is that, while the frequencies of
base $b$ digits are uniformly distributed on the finite set
$\{0,1,\dots,b-1\}$,  continued fraction digits 
can take any positive integer value, and 
their frequencies, given by \eqref{eq:gk-distribution-def}, are non-uniform.

A second difference is that
consecutive continued fraction digits are not independent under the
Gauss--Kuzmin measure\footnote{%
As an aside, we note that the question whether or not the ``digits''
in a generalized digital expansion are independent (under the 
appropriate ``natural'' measure) can be quite subtle. For example, the digits  
in L\"uroth expansions---a very different type of
construction than decimal and base $b$ expansions---are independent, 
while those in 
$\beta$-expansions---which extend the usual base $b$ expansions 
to non-integer bases $\beta$---are \emph{not} independent; see
Sections 2.2 and 2.4 in \cite{dajani-kraaikamp-book}.}.
That is, in
general we have
$P_{GK}((a_1,\allowbreak \dots,a_n))\allowbreak
\not=P_{GK}(a_1)\dots P_{GK}(a_n)$.  For example, 
by \eqref{eq:example-P314} the string 
$(3,1,4)$ occurs with frequency
$P_{GK}((3,1,4))=-\log_2(1-1/551)=0.002620\dots$, while
from \eqref{eq:gk-distribution-def} we get
\begin{align*}
P_{GK}(3)P_{GK}(1)P_{K}(4)
&=
\log_2 \left(1+\frac1{3\cdot 5}\right)
\log_2 \left(1+\frac1{1\cdot 3}\right)
\log_2 \left(1+\frac1{4\cdot 6}\right)
\\
&=0.002275\dots\ .
\end{align*}

A third difference is that frequencies of strings of continued fraction
digits depend on the order in which these digits occur in the string;
that is, different permutations of the same string \emph{in general}
have different frequencies.  For example, by \eqref{eq:example-P314} the
string $(3,1,4)$ occurs with frequency $-\log_2(1-1/551)$, while an
analogous calculation shows that the string $(3,4,1)$ occurs with the
slightly smaller frequency $-\log_2(1-1/608)$.  It is this dependency of
the frequency of a string of digits in continued fraction expansions on
the order of the digits in the string that we will focus on in 
this paper. 

There is a notable
exception to the dependency on the order of the digits: 
the reverse of a string \emph{always} has the same
frequency as the string itself. 
Specifically, given 
a finite string $\bfa=(a_1,\dots,a_n)$, let $\bfar$ denote the string
obtained by reversing the order of the digits $a_i$, i,e., 
\begin{equation}
\notag
\bfar=(a_n,\dots,a_1).
\end{equation}
Then the following result holds.

\begin{proposition}[Symmetry Property]
\label{lem:symmetry-property}
For all finite strings $\bfa=(a_1,\dots,a_n)$ of positive integers we have
\begin{equation}
\label{eq:symmetry}
\mu_{GK}(I(\bfa))=\mu_{GK}(I(\bfar))
\end{equation}
and thus
\begin{equation}
\label{eq:symmetry-alt}
P_{GK}(\bfa)=P_{GK}(\bfar).
\end{equation}
\end{proposition}

This surprising property of the Gauss--Kuzmin measure is an elementary
consequence of the definitions \eqref{eq:mugk-def}  and
\eqref{eq:Ibfa-def} of the Gauss--Kuzmin measure and the sets 
$I(\bfa)$. The property has been observed
in the literature (see, e.g., \cite[p.~189]{levy1929} and 
\cite[p.~430]{vandehey2016}),
though does not seem to be widely known; we will provide a proof in
Section \ref{sec:background} (see Lemma 
\ref{lem:gk-formulas})\footnote{%
While the existence of an identity such as \eqref{eq:symmetry}
relating the Gauss--Kuzmin measures of 
$I(\bfa)$ and $I(\bfar)$ may seem surprising at first glance, it can 
be explained using ideas from \cite[Section 5.3]{dajani-kraaikamp-book}.}.

Table \ref{tab:table1} illustrates the dependence of the frequencies
$P_{GK}(\bfa)$ on the permutations of the string $\bfa$ as well as
the symmetry property  \eqref{eq:symmetry-alt}.
The table shows 
the frequencies of all six permutations of the string
$(3,1,4)$.
As predicted by 
\eqref{eq:symmetry-alt}
the reverse of the string $(3,1,4)$, i.e.,  the permutation  $(4,1,3)$, 
has the same frequency as the string itself, and the same holds for
the pairs $\{(1,3,4), (4,3,1)\}$ and $\{(3,4,1),(1,4,3)\}$.

\begin{table}[H]
\begin{center}
\renewcommand{\arraystretch}{2.5}
\addtolength{\tabcolsep}{5pt}
\begin{tabular}{|c|c|c|}
\hline
String $\bfa$ & Interval $I(\bfa)$ & Frequency $P_{GK}(\bfa)$
\\
\hline
 $(1,3,4)$ & $\left(\dfrac{16}{21},\dfrac{13}{17}\right]$ & 
 $-\log_2\left(1-\dfrac{1}{629}\right)$ \\
 $(1,4,3)$ & $\left(\dfrac{17}{21},\dfrac{13}{16}\right]$ & 
 $-\log_2\left(1-\dfrac{1}{608}\right)$ \\
 $(3,1,4)$ & $\left(\dfrac{6}{23},\dfrac{5}{19}\right]$ & 
 $-\log_2\left(1-\dfrac{1}{551}\right)$ \\
 $(3,4,1)$ & $\left(\dfrac{9}{29},\dfrac{5}{16}\right]$ &
 $-\log_2\left(1-\dfrac{1}{608}\right)$ \\
 $(4,1,3)$ & $\left(\dfrac{5}{24},\dfrac{4}{19}\right]$ & 
 $-\log_2\left(1-\dfrac{1}{551}\right)$ \\
 $(4,3,1)$ & $\left(\dfrac{7}{30},\dfrac{4}{17}\right]$ & 
 $-\log_2\left(1-\dfrac{1}{629}\right)$ \\
 \hline
 \end{tabular}
 \end{center}
\caption{The frequencies $\mu_{GK}(I(\bfa))$ 
of the $6$ permutations $\bfa$ of the string $(3,1,4)$.}
\label{tab:table1}
\end{table}

In our first result we show that the symmetry property
\eqref{eq:symmetry} in fact characterizes the Gauss--Kuzmin measure
$\mu_{GK}$.

\begin{theorem}[Characterization of the Gauss--Kuzmin measure]
\label{thm:gk-characterization}
Let $S\subset \NN$ be an infinite set of positive integers. 
Let $\mu$ be a probability measure on $[0,1]$  with continuous density
function satisfying
\begin{equation}
\label{eq:symmetry-general}
\mu(I(\bfa))=\mu(I(\bfar))
\end{equation}
for all strings $\bfa=(a_1,\dots,a_n)$ of positive integers 
of length $n\in S$. Then
$\mu=\mu_{GK}$, i.e., $\mu$ is the measure with density function
$f(x)=(1/\log
2)(1+x)^{-1}$.
\end{theorem}

This characterization is best-possible\footnote{But see Remark
\ref{rem:thm1refinements} for comments on possible refinements of the
statement of the theorem.}
in the sense that if we impose 
the symmetry property \eqref{eq:symmetry-general} only for strings
$\bfa$ of length from a \emph{finite} set $S$, the conclusion
$\mu=\mu_{GK}$ need not hold:

\begin{theorem}[Optimality of the characterization]
\label{thm:gk-characterization-optimality}
Let $N\in \NN$ be given. Then there exists 
a probability measure $\mu$ on $[0,1]$ with continuous density
function that satisfies \eqref{eq:symmetry-general} 
for all strings $\bfa=(a_1,\dots,a_n)$ of positive integers 
of length $n\le N$, but $\mu\not=\mu_{GK}$.
\end{theorem}

In the example shown in Table \ref{tab:table1}, it is the case that two
different permutations of $(3,1,4)$ have the same frequency \emph{if and only
if} one is the reverse of the other.   We next explore the question
whether this holds for more general strings. 

Given a permutation $\sigma$ of the indices $1,2,\dots,n$ and a string
$\bfa=(a_1,\dots,a_n)$, 
let $\sigma(\bfa)$ denote the permutation of $\bfa=(a_1,\dots,a_n)$ 
induced by $\sigma$, i.e.,
\begin{equation}
\notag
\sigma(\bfa)=(a_{\sigma(1)},\dots,a_{\sigma(n)}).
\end{equation}

\begin{definition}[Non-trivial symmetries]
\label{def:nontrivial-symmetries}
Let $\bfa$ be a string of positive integers of 
length $n$. We say that the string $\bfa$ has a
\emph{nontrivial symmetry} if there exists a permutation $\sigma$ of
$1,\dots,n$ with $\sigma(\bfa)\not=\bfa$ and $\sigma(\bfa)\not=\bfar$ 
such that
\begin{equation}
\notag
P_{GK}(\sigma(\bfa))=P_{GK}(\bfa).
\end{equation}
\end{definition}

From Table \ref{tab:table1} we see that the string $(3,1,4)$ has \emph{no}
nontrivial symmetries. This raises the question of whether the same 
holds for more general strings.  For strings of length $2$, this is
trivially the case as the only permutations of such a string 
are the string itself and
its reverse.  The question becomes nontrivial for strings of length 
$3$ and larger.  We prove the following result. 

\begin{theorem}[Strings with nontrivial symmetries]
\label{thm:nontrivial-symmetries}
\mbox{}
\begin{itemize}
\item[(i)]
There exists no  string  $\bfa$ of length $3$
with a nontrivial symmetry.
\item[(ii)]
For each $n\ge 4$ there exists an infinite, 
$\lfloor (n-2)/2\rfloor$-parameter, family of strings 
$\bfa$ of length $n$ that have a nontrivial symmetry.
\end{itemize}
\end{theorem}

In Section \ref{sec:numerical-data} we present numerical data 
suggesting that strings with nontrivial symmetries are quite rare
in the following sense.

\begin{conjecture}[Strings with nontrivial symmetries]
\label{conj:nontrivial-symmetries}
Let $n\ge4$ be given. Then 
\begin{equation}
\notag
\lim_{N\to\infty}\frac1{N^n}\#\left\{\bfa\in \{1,\dots,N\}^n: 
\text{$\bfa$ has a nontrivial symmetry}\right\}=0.
\end{equation}
\end{conjecture}

This may be interpreted as saying that almost all strings $\bfa$ have
no nontrivial symmetries.

\section{Background on Continued Fractions}
\label{sec:background}

\subsection{Continued Fraction Basics}
We begin by recalling some key definitions and facts from the elementary
theory of continued fractions.  Details and proofs can be found, for
example,  in
\cite[Chapter 2]{neverending-fractions-book}, 
\cite[Section 1.3]{dajani-kraaikamp-book},
\cite[Chapter
9]{hardy-wright2008}, \cite[Chapters I--II]{khinchin-book}, and
\cite[Chapter 5]{niven1956}. 

A \emph{continued fraction} is a finite or infinite expression of the form
\begin{equation}
\label{eq:cf-def1}
a_0+\cfrac{1}{a_1+\cfrac{1}{a_2+\cfrac{1}{\ddots}}} = [a_0;
a_1,a_2,\dots],
\end{equation}
where $a_0$ is an arbitrary integer, and $a_1,a_2,\dots $ are positive
integers.

Clearly, any \emph{finite} (i.e., terminating) continued fraction
represents a rational number. Conversely, any rational number can be
represented as a finite continued fraction $[a_0;a_1,\dots,a_n]$.
There is a slight ambiguity in this representation due to the identity 
\begin{equation}
\label{eq:cf-nonuniqueness}
[a_0; a_1,\dots,a_n]=[a_0;a_1,\dots,a_n-1,1]
\quad\text{if $a_n>1$.}
\end{equation}
This ambiguity can be eliminated by requiring the last digit in the
representation to be strictly greater than $1$.

An \emph{infinite} continued fraction is defined as the limit, as
$n\to\infty$,  of the finite continued fractions obtained by truncating
the given infinite continued fraction after $n$ terms:
\begin{equation}
\label{eq:cf-infinite-def}
[a_0;a_1,a_2,\dots]=
\lim_{n\to\infty} [a_0;a_1,a_2,\dots,a_n].
\end{equation}
It is known (see, e.g., \cite[Theorem 5.11]{niven1956})
that any \emph{irrational} real number has a unique infinite
continued fraction expansion; that is, there exists a unique integer
$a_0$ and and unique positive integers  
$a_1,a_2,\dots$ such that
\begin{equation}
\notag
x= \lim_{n\to\infty} [a_0;a_1,a_2,\dots,a_n].
\end{equation}
Conversely, given any integer $a_0$ and any sequence
$a_1,a_2,\dots$ of positive integers, the limit
\eqref{eq:cf-infinite-def} exists and represents an irrational number.

The \emph{convergents} of a (finite or infinite) continued fraction 
$[a_0;a_1,a_2,\dots]$
are the rational numbers obtained by truncating the continued fraction 
after finitely many terms. The $n$th convergent is the continued
fraction truncated at $a_n$ and is 
denoted by $p_n/q_n$; that is, $p_n$ and $q_n$ are integers
satisfying  
\begin{equation}
\notag
[a_0;a_1,\dots,a_n]=\frac{p_n}{q_n},
\end{equation}
with the convention that
\begin{equation}
\label{eq:pn-qn-conventions}
\begin{cases}
p_0=a_0,\ q_0=1,&
\\
p_n\in\ZZ,\ q_n\in\NN,\quad (p_n,q_n)=1& (n\ge 1).
\end{cases}
\end{equation}
The numbers $p_n$ and $q_n$ can be computed recursively in terms of the digits
$a_i$, or equivalently through the matrix identity (see, for example,
{\cite[Section 1.3.2]{dajani-kraaikamp-book}}
or---in a slightly different,
but equivalent form---\cite[\S 5]{perron-book})
\begin{equation}
\label{eq:key0}
\begin{pmatrix} 1&a_0\\ 0 & 1\end{pmatrix}
\mat{a_1}\dots \mat{a_n}
=\begin{pmatrix} p_{n-1} & p_{n} \\ q_{n-1} &q_{n}\end{pmatrix}
\quad (n\ge 1).
\end{equation} 

Taking the determinant on both sides of \eqref{eq:key0} yields the
identity
\begin{equation}
\label{eq:determinant-formula}
p_nq_{n-1}-q_np_{n-1}=(-1)^{n-1}
\quad (n\ge 1).
\end{equation}

In the remainder of this paper we will focus on 
continued fractions with leading term $a_0=0$, i.e., continued fractions 
representing numbers in the unit interval $[0;1)$. 
In this case, we suppress the term $a_0(=0)$ in the continued fraction 
notation \eqref{eq:cf-def1} and thus write
\begin{equation}
\notag
[a_1,a_2,\dots]= [0; a_1,a_2,\dots]
=\cfrac{1}{a_1+\cfrac{1}{a_2+\cfrac{1}{\ddots}}} \ .
\end{equation}
Note that when $a_0=0$, we have,  by \eqref{eq:pn-qn-conventions},
\begin{equation}
\notag
p_0=0,\quad q_0=1.
\end{equation}

Given a finite string $\bfa=(a_1,\dots,a_n)$ of positive integers, we let 
\begin{equation}
\notag
[\bfa]=[a_1,\dots,a_n]
\left(=[0;a_1,\dots,a_n]\right)
\end{equation}
denote the continued fraction with digits given by this string. 
Given two finite strings 
$\bfa=(a_1,\dots,a_n)$ and $\bfb=(b_1,\dots,b_m)$ of positive integers,
we denote by $(\bfa,\bfb)=(a_1,\dots,a_n,b_1,\dots,b_m)$ the
concatenation of these strings and by 
$[\bfa,\bfb]=[a_1,\dots,a_n,b_1,\dots,b_m]$
the corresponding continued fraction.
In the case when the string $\bfb=(b)$ consists of a single digit $b$, we 
suppress the parentheses around $b$ and write $(\bfa,b)$ and $[\bfa,b]$
instead of $(\bfa,(b))$ and $[\bfa,(b)]$, respectively;
$(b,\bfa)$ and $[b,\bfa]$ are to be interpreted analogously.

\subsection{Convergent Matrices}
It will be convenient to encode the last two convergents of a finite
continued fraction as a $2\times 2$ matrix, defined as follows. 

\begin{definition}[Convergent matrix]
\label{def:convergent-matrix}
Given a finite string $\bfa=(a_1,\dots,a_n)$ of positive integers, 
the \emph{convergent matrix} associated with this string (or with the
continued fraction, $[\bfa]$, defined by this string)
is  the matrix $C(\bfa)$ defined by 
\begin{equation}
\label{eq:def-convergent-matrix}
C(\bfa)=
\begin{pmatrix} p_{n-1} & p_n\\ q_{n-1} & q_n\end{pmatrix},
\end{equation}
where $p_{n-1}/q_{n-1}$ and $p_n/q_n$ denote the 
$(n-1)$th and $n$th convergents 
of the continued fraction $[\bfa]=[a_1,\dots,a_n]$.
\end{definition}

Note that the matrix $C(\bfa)$ defined in \eqref{eq:def-convergent-matrix}
is exactly the matrix appearing on the right side of the identity
\eqref{eq:key0}.  Since, by our assumption, $a_0=0$, 
the first matrix on the left of 
\eqref{eq:key0} reduces to the identity matrix, so
\eqref{eq:key0} simplifies to
\begin{equation}
\label{eq:key1}
C(\bfa)=\mat{a_1}\dots \mat{a_n}.
\end{equation}

We next derive explicit formulas for the convergent matrix $C(\bfar)$ 
associated with the reverse $\bfar$ of a string $\bfa$, and the
convergent matrices $C(\bfa,t)$ and $C(t,\bfa)$ associated with the  
strings $(\bfa,t)$ and $(t,\bfa)$, obtained by appending or prepending a
single digit $t$ to the string $\bfa$.

\begin{lemma}[Explicit formulas for convergent matrices]
\label{lem:convergent-matrices}
Let $\bfa=(a_1,\dots,a_n)$
be a string of positive integers, 
and denote the convergent matrix of
$\bfa$ by  
\begin{equation}
\label{eq:convergent-matrix-alt}
C(\bfa)=
\begin{pmatrix} p'&p\\q'&q \end{pmatrix}
\end{equation}
(so that $(p,q)=(p_n,q_n)$ and $(p',q')=(p_{n-1},q_{n-1})$).
Then we have:
\begin{itemize}
\item[(i)]
$\displaystyle C(\bfar)=
\begin{pmatrix} p' & q'\\p&q\end{pmatrix}$;

\item[(ii)]
$\displaystyle C(\bfa,t)=
\begin{pmatrix} p & p'+tp\\ q&q'+tq \end{pmatrix}
\quad (t\in\NN)$;

\item[(iii)]
$\displaystyle 
C(t,\bfa)=
\begin{pmatrix} q' & q \\ p'+tq'& p+tq \end{pmatrix}
\quad (t\in\NN)$;

\item[(iv)]
$\displaystyle
C(t,\bfar)=
\begin{pmatrix} p & q \\ p'+tp& q'+tq \end{pmatrix}
\quad (t\in\NN)$.
\end{itemize}
\end{lemma}

\begin{proof}
(i).
Taking the transpose on both sides of the identity \eqref{eq:key1}
and noting that the matrices $\mat{a_i}$ are symmetric, we obtain
\begin{align*}
C(\bfa)^T&=\left(\mat{a_1}\dots \mat{a_n}\right)^T
=\mat{a_n}^T\dots \mat{a_1}^T
\\
& =\mat{a_n}\dots \mat{a_1} =C(\bfar).
\end{align*}
Since 
\begin{equation*}
C(\bfa)^T=
\begin{pmatrix}p'&p\\q'&q\end{pmatrix}^T
=\begin{pmatrix} p' & q'\\p&q\end{pmatrix},
\end{equation*}
this proves (i).

\medskip

To prove (ii)--(iv), let $t$ be a positive integer and apply the
identity \eqref{eq:key1} to the strings $(\bfa,t)=(a_1,\dots,a_n,t)$,
$(t,\bfa)=(t,a_1,\dots,a_n)$, and $(t,\bfar)=(t,a_n,\dots,a_1)$.
We obtain 
\begin{align*}
C(\bfa,t)
&=\mat{a_1}\dots\mat{a_n}\mat{t}  
=C(\bfa) 
\begin{pmatrix} 0&1\\1&t\end{pmatrix}
\\
&=\begin{pmatrix} p'&p\\q'&q\end{pmatrix}
\begin{pmatrix} 0&1\\1&t\end{pmatrix}
=
\begin{pmatrix} p & p'+tp\\ q&q'+tq \end{pmatrix},
\\[1ex]
C(t,\bfa)
&=\mat{t}\mat{a_1}\dots\mat{a_n}
= \begin{pmatrix} 0&1\\1&t\end{pmatrix}
C(\bfa)
\\
&=
\begin{pmatrix} 0&1\\1&t\end{pmatrix}
\begin{pmatrix} p'&p\\q'&q\end{pmatrix}
=
\begin{pmatrix} q' & q \\ p'+tq'& p+tq \end{pmatrix},
\\[1ex]
C(t,\bfar)
&=\mat{t}\mat{a_n}\dots\mat{a_1}
= \begin{pmatrix} 0&1\\1&t\end{pmatrix} C(\bfar)
\\
&=
\begin{pmatrix} 0&1\\1&t\end{pmatrix}
\begin{pmatrix} p'&q'\\p&q\end{pmatrix}
=
\begin{pmatrix} p & q \\ p'+tp& q'+tq \end{pmatrix},
\end{align*}
as claimed.
\end{proof}

\begin{remark}
\label{rem:symmetry-property}
The identity in part (i) of the lemma relates the last two
convergents of a finite continued fraction corresponding to a string
$\bfa$ to the last two convergents of the continued fraction
corresponding to the reverse string $\bfar$. 
Identities of this type have long been known in the literature (see,
e.g., \cite[\S 11]{perron-book}, \cite[p.~189]{levy1929}, 
and \cite[Theorem 6]{khinchin-book}).
It is this identity that
lies at the root of the symmetry property \eqref{eq:symmetry}. By the
same token, the (apparent) lack of analogous identities for permutations
other than the reversal may explain why for
``most'' strings $\bfa$, the reversal $\bfar$ seems to be the only
permutation of $\bfa$ under which the Gauss--Kuzmin measure is invariant.
\end{remark}

\subsection{Fundamental Intervals}

Let $\bfa=(a_1,\dots,a_n)$ be a string of positive integers, and let 
$C(\bfa)=\left(\begin{smallmatrix} p'&p\\q'&q\end{smallmatrix}\right)$ 
be the associated convergent matrix (cf.  \eqref{eq:convergent-matrix-alt}).
We begin by obtaining an explicit formula for the set $I(\bfa)$ defined
in \eqref{eq:Ibfa-def} in terms of the entries $p',p,q',q$ 
of this matrix. 
By definition $I(\bfa)$ is the set of real
numbers $x\in[0,1)$ whose continued fraction expansion begins with the
digits of $\bfa$, 
i.e., the set of real numbers of the form $x=[a_1,\dots,a_n]$ or
$x=[a_1,\dots,a_n,*]$, 
with the asterisk denoting one or more positive integers. 
It follows easily from the definitions \eqref{eq:cf-def1} and
\eqref{eq:cf-infinite-def} of finite and infinite continued fractions 
that this set is an interval
with endpoints 
\begin{equation}
\notag
[a_1,\dots,a_n]=\frac{p_n}{q_n}=\frac{p}{q}
\end{equation}
and 
\begin{equation}
\notag
[a_1,\dots,a_n+1]=[a_1,\dots,a_n,1]=\frac{p'+p}{q'+q},
\end{equation}
where the first of these two endpoints is included, and the second
excluded.  We call this interval a \emph{fundamental interval}.  
Using the notation
\begin{equation}
\notag
\Fundint{\alpha,\beta}=
\begin{cases} 
[\alpha,\beta) & \text{if $\alpha< \beta$,}
\\
(\beta,\alpha] & \text{if $\alpha>\beta$,}
\end{cases}
\end{equation}
we thus have (cf. \cite[(3.6)]{neverending-fractions-book} or 
\cite[Ex.~1.3.15]{dajani-kraaikamp-book})
\begin{equation}
\label{eq:fundint-formula0}
I(\bfa)=
\Fundint{\frac{p}{q},\frac{p'+p}{q'+q}}.
\end{equation}
Since
\begin{equation}
\notag
\frac{p'+p}{q'+q}-\frac{p}{q} = \frac{(p'+p)q-p(q'+q)}{q(q'+q)}
=\frac{p'q-pq'}{q(q'+q)}=\frac{(-1)^n}{q(q'+q)},
\end{equation}
where the last step follows from \eqref{eq:determinant-formula},
\eqref{eq:fundint-formula0} can be written as 
\begin{equation}
\label{eq:fundint-formula1}
I(\bfa)=
\Fundint{\frac{p}{q},\frac{p}{q}+\frac{(-1)^n}{q(q'+q)}}.
\end{equation}

Next, we derive explicit formulas for the fundamental intervals
associated with strings of the form $\bfar$, $(\bfa,t)$, $(t,\bfa)$, and
$(t,\bfar)$.

\begin{lemma}
[Explicit formulas for fundamental intervals]
\label{lem:fundamental-intervals1}
Let $\bfa=(a_1,\dots,a_n)$
be a string of $n$ positive integers with convergent matrix
\eqref{eq:convergent-matrix-alt}. Then we have:
\begin{itemize}
\item[(i)]
$\displaystyle I(\bfar)=
\Fundint{\frac{q'}{q},\frac{q'}{q}+\frac{(-1)^n}{q(p+q)}}$;

\item[(ii)]
$\displaystyle I(\bfa,t)=
\Fundint{\frac{p'+tp}{q'+tq},
\frac{p'+tp}{q'+tq}+\frac{(-1)^{n+1}}{(q'+tq)(q'+(t+1)q)}}
\quad(t\in\NN)$;

\item[(iii)]
$\displaystyle I(t,\bfa)=
\Fundint{\frac{q}{p+tq},\frac{q}{p+tq}+\frac{(-1)^{n+1}}{(p+tq)(p'+p +
t(q'+q))}} \quad(t\in\NN)$;

\item[(iv)]
$\displaystyle I(t,\bfar)=
\Fundint{\frac{q}{q'+tq},\frac{q}{q'+tq}+
\frac{(-1)^{n+1}}{(q'+tq)(p'+q'+t(p+q))}}
\quad(t\in\NN)$.

\end{itemize}

\end{lemma}

\begin{proof}
By part (i) of Lemma \ref{lem:convergent-matrices}, the 
convergent matrix 
$C(\bfar)$
associated with the reverse string $\bfar$
is obtained from the convergent matrix 
$C(\bfa) =\left(\begin{smallmatrix} p'&p\\ q'&q\end{smallmatrix}\right)$
by interchanging $p$ and $q'$, i.e., by making the substitutions 
\begin{equation}
\notag
p\to q',\quad q'\to p.
\end{equation}
Performing the same substitutions in the formula 
\eqref{eq:fundint-formula1}
for the interval $I(\bfa)$ then yields
the formula for $I(\bfar)$ asserted in part (i) of the lemma.

The formulas for $I(\bfa,t)$, $I(t,\bfa)$, and $I(t,\bfar)$
in parts (ii)--(iv) of 
the lemma can be obtained similarly using the formulas for
$C(\bfa,t)$, $C(t,\bfa)$, and $C(t,\bfar)$  from Lemma 
\ref{lem:convergent-matrices} and making the substitutions 
\begin{align}
\notag
&p'\to p,\quad q'\to q,\quad p\to p'+tp,\quad q\to q'+tq,
\\
\notag
&p'\to q',\quad q'\to p'+tq',\quad p\to q,\quad q\to p+tq,
\\
\notag
&p'\to p,\quad q'\to p'+tp,\quad p\to q,\quad q\to q'+tq,
\end{align}
and also replacing $n$ by $n+1$ to account for the additional digit $t$ in
the strings $(\bfa,t)$, $(t,\bfa)$, and $(t,\bfar)$. 
\end{proof}

\subsection{The Gauss--Kuzmin Measure}

Using the formula \eqref{eq:fundint-formula1}
for $I(\bfa)$ 
we can compute 
the Gauss--Kuzmin probabilities 
(see \eqref{eq:mugk-def} and \eqref{eq:Pbfa-formula})
\begin{equation}
P_{GK}(\bfa)=\mu_{GK}(I(\bfa))=\frac1{\log 2}\int_{I(\bfa)}
\frac1{1+x}\, dx
\end{equation}
in terms of the entries $p,q,p',q'$ of the convergent matrix associated
with the string $\bfa$.

\begin{lemma}[Explicit formula for the Gauss--Kuzmin measure]
\label{lem:gk-formulas}
Let $\bfa=(a_1,\dots,a_n)$
be a string of positive integers, with convergent matrix
given by \eqref{eq:convergent-matrix-alt}.  
Then we have:
\begin{equation}
\label{eq:gk-formula}
P_{GK}(\bfa) = P_{GK}(\bfar) 
= \left|\log_2\left(1+\frac{(-1)^n}{(p+q)(q'+q)}\right)\right|.
\end{equation}
In particular, the Gauss--Kuzmin measure satisfies the symmetry property 
\eqref{eq:symmetry-alt}.
\end{lemma}

\begin{proof}
By \eqref{eq:fundint-formula1}, $I(\bfa)$ is an
interval with endpoints
\begin{equation}
\label{eq:Ibfa-endpoints-formula}
\alpha=\frac{p}{q}, \quad 
\beta=\frac{p}{q}+\frac{(-1)^n}{q(q'+q)}.
\end{equation}
Hence,
\begin{align}
\notag
P_{GK}(\bfa)
&=\mu_{GK}(I(\bfa))=
\mu_{GK}(\fundint{\alpha,\beta})
\\
\label{eq:pgk-calculation1}
&=\frac1{\log 2}
\left|\int_{\alpha}^\beta \frac1{1+x}\,dx\right|
=\left|\log_2\frac{1+\beta}{1+\alpha}\right|.
\end{align}
Using \eqref{eq:Ibfa-endpoints-formula} we get
\begin{align}
\label{eq:pgk-calculation2}
\frac{1+\beta}{1+\alpha} 
&=\frac{1+\frac{p}{q}+\frac{(-1)^n}{q(q'+q)}}
{1+\frac{p}{q}}
=\frac{(p+q)(q'+q)+(-1)^n}{(p+q)(q'+q)}
=1+\frac{(-1)^n}{(p+q)(q'+q)}.
\end{align}
Substituting \eqref{eq:pgk-calculation2}
into \eqref{eq:pgk-calculation1} yields the desired formula for 
$P_{GK}(\bfa)$.

To obtain an analogous formula for $P_{GK}(\bfar)$,
note that the denominator in the formula
\eqref{eq:gk-formula}
for $P_{GK}(\bfa)$, i.e., the expression $(p+q)(q'+q)$, 
is the product of the sum of the entries in the second column of $C(\bfa)$
with the sum of the entries in the second row of $C(\bfa)$.
If $\bfa$ is replaced by $\bfar$, then by Lemma 
\ref{lem:convergent-matrices}, the entries in the second column of
$C(\bfar)$ become $(q',q)$ and thus have sum $q'+q$, while the entries
in the second row become $(p,q)$ and thus have sum $p+q$. The 
product of these two sums is $(q'+q)(p+q)$. The latter expression 
is identical to the corresponding expression, 
$(p+q)(q'+q)$, for the matrix $C(\bfa)$. Since the strings $\bfa$ and
$\bfar$ both have the same length $n$, it follows that
$P_{GK}(\bfar)=P_{GK}(\bfa)$.
\end{proof}


\section{Proof of Theorem \protect \ref{thm:gk-characterization}} 
\label{sec:thm1}

Throughout this section
we assume $\mu$ is a probability measure on $[0,1]$ with continuous
density function $f(x)$, so that
\begin{equation}
\label{eq:proof1.1-mu-def}
\mu(I)=\int_I f(x)dx
\end{equation}
for any interval $I\subseteq [0,1]$ and 
\begin{equation}
\notag
\mu([0,1])=\int_0^1f(x)dx=1.
\end{equation}

The crux of the proof lies in the following result.

\begin{lemma}
\label{lem:proof1.1-lemma1}
Let $\bfa=(a_1,\dots,a_n)$ be a string of positive integers,  and let
\begin{equation}
\notag
r=[a_1,\dots,a_n]
\end{equation}
be the rational number represented by the continued fraction $[\bfa]$.
Suppose that the measure
$\mu$ satisfies the symmetry property for all strings of the form
$(\bfa,t)=(a_1,\dots,a_n,t)$, $t\in\NN$; i.e., suppose that
\begin{equation}
\label{eq:proof1.1-symmetry1}
\mu(I(\bfa,t))=\mu(I(t,\bfar))\quad (t\in\NN).
\end{equation}
Then we have
\begin{equation}
\notag
f(r)=\frac{f(0)}{1+r}.
\end{equation}
\end{lemma}

\begin{proof}
By Lemma \ref{lem:fundamental-intervals1} we have, for any
$t\in\NN$,
\begin{align}
\notag
I(\bfa,t) &= \Fundint{\alpha(t),\alpha(t)+(-1)^{n+1}\delta(t)},
\\
\notag
I(t,\bfar)&=\Fundint{\beta(t),\beta(t)+(-1)^{n+1}\epsilon(t)},
\end{align}
where 
\begin{align}
\label{eq:proof1.1-alpha-def}
\alpha(t)&=\frac{p'+tp}{q'+tq},
\quad \delta(t)=\frac1{(q'+tq)(q'+(t+1)q)},
\\
\label{eq:proof1.1-beta-def}
\beta(t)&=\frac{q}{q'+tq},
\quad \epsilon(t)=\frac1{(q'+tq)(p'+q'+t(p+q))}.
\end{align}
(Recall that $p'/q'$ and $p/q$ denote the last two convergents of the
continued fraction $[a_1,\dots,a_n]$, so that 
$[a_1,\dots,a_{n-1}]=p'/q'$ and $[a_1,\dots,a_n]=p/q=r$, 
with the convention that, when $n=1$,
$(p',q')=(p_0,q_0)=(0,1)$.)

Using \eqref{eq:proof1.1-mu-def} and the mean value theorem for
integrals, it follows
that
\begin{align}
\label{eq:proof1.1-muIbfat}
\mu(I(\bfa,t))&=\left|\int_{\alpha(t)}^{\alpha(t)+(-1)^{n+1}\delta(t)}
f(x)dx\right| =f(\xi(t))\delta(t),
\\
\label{eq:proof1.1-muItbfar}
\mu(I(t,\bfar))&=\left|\int_{\beta(t)}^{\beta(t)+(-1)^{n+1}\epsilon(t)}
f(x)dx\right| =f(\eta(t))\epsilon(t),
\end{align}
where $\xi(t)$ and $\eta(t)$ are real numbers in $[0,1]$ satisfying
\begin{align}
\label{eq:proof1.1-xi}
|\xi(t)-\alpha(t)|&\le \delta(t),
\\
\label{eq:proof1.1-eta}
|\eta(t)-\beta(t)|&\le \epsilon(t).
\end{align}

Now note that, as $t\to\infty$, we have, by 
\eqref{eq:proof1.1-alpha-def}
and \eqref{eq:proof1.1-beta-def}, 
\begin{align*}
\alpha(t)&= \frac{p+p'/t}{q+q'/t}\to \frac{p}{q}=r,
\quad \delta(t)\to 0,
\\
\beta(t)&=\frac{p'}{p+tp'}\to 0,
\quad \epsilon(t)\to 0.
\end{align*}
Using 
\eqref{eq:proof1.1-xi} and \eqref{eq:proof1.1-eta}, it follows that 
\begin{align*}
\lim_{t\to\infty}\xi(t)&=\lim_{t\to\infty}\alpha(t)=r,
\\
\lim_{t\to\infty}\eta(t)&=\lim_{t\to\infty}\beta(t)=0,
\end{align*}
and hence, by the continuity of $f(x)$,
\begin{align}
\label{eq:proof1.1-fxi-limit}
\lim_{t\to\infty}f(\xi(t))&=f(r),
\\
\label{eq:proof1.1-feta-limit}
\lim_{t\to\infty}f(\eta(t))&=f(0).
\end{align}

On the other hand, the symmetry assumption \eqref{eq:proof1.1-symmetry1}
along with the formulas \eqref{eq:proof1.1-muIbfat}
and \eqref{eq:proof1.1-muItbfar} imply that 
\begin{equation*}
f(\xi(t)) \delta(t)= f(\eta(t))\epsilon(t),
\end{equation*}
and hence 
\begin{equation*}
f(\xi(t))=f(\eta(t))\frac{\epsilon(t)}{\delta(t)}.
\end{equation*}
Letting $t\to\infty$ on both sides of the latter identity, we obtain, in
view of \eqref{eq:proof1.1-fxi-limit} and
\eqref{eq:proof1.1-feta-limit}, 
\begin{align*}
f(r)&=f(0)\lim_{t\to\infty}\frac{\epsilon(t)}{\delta(t)}
\\
&=f(0)\lim_{t\to\infty}
\frac{(q'+tq)(q'+(t+1)q)}
{(q'+tq)(p'+q'+t(p+q))}
=\frac{f(0)}{1+p/q}=\frac{f(0)}{1+r},
\end{align*}
as claimed.
\end{proof}

\begin{proof}[Proof of Theorem \ref{thm:gk-characterization}]
Assume now that $\mu$
satisfies the full strength of the hypothesis of Theorem 
\ref{thm:gk-characterization}, i.e., that the symmetry property 
$\mu(I(\bfa))=\mu(I(\bfar))$ holds
for all finite strings $\bfa$ of positive integers whose length belongs to a
given infinite set $S$.  
By Lemma \ref{lem:proof1.1-lemma1} we then have
\begin{equation}
\notag
f(r)=\frac{f(0)}{1+r} 
\end{equation}
for all rational numbers $r$ of the form
\begin{equation}
\label{eq:proof1.1-r-formula}
r=[a_1,\dots,a_n],\quad n+1\in S,\quad a_i\in\NN \quad (i=1,\dots,n).
\end{equation}
To complete the proof, it
remains to show that the set of numbers $r$ of the form 
\eqref{eq:proof1.1-r-formula}
is
dense in the interval $[0,1]$. Indeed, if this set is dense in $[0,1)$,
then the continuity of $f$ implies that the relation
$f(x)=f(0)/(1+x)$ holds for \emph{all} real $x\in(0,1)$, while the assumption 
that $f$ is a probability density function forces 
$f(0)=1/\log 2$. Hence, $f$ is the density function of the
Gauss--Kuzmin measure, and we conclude $\mu=\mu_{GK}$, as claimed.

To prove the above claim, it is enough to show that every
\emph{irrational} number in $(0,1)$ can be approximated arbitrarily
closely by numbers of the form \eqref{eq:proof1.1-r-formula}.  
Fix an irrational number $x\in(0,1)$, let 
\begin{equation}
\notag
x=[a_1,a_2,\dots]
\end{equation}
be the (infinite) continued fraction expansion of $x$, and let
\begin{equation}
\notag
r_n=[a_1,\dots,a_n]
\end{equation}
be the $n$th convergent of this continued fraction. Note that $r_n$ is
of the form \eqref{eq:proof1.1-r-formula} whenever $n+1\in S$. Since 
$\lim_{n\to\infty}r_n=x$ and the set $S$ is infinite, it follows that 
$x$ can be approximated arbitrarily closely by numbers $r_n$ with
$n+1\in S$, and hence by numbers of the form
 \eqref{eq:proof1.1-r-formula}.  
This completes the proof of Theorem \ref{thm:gk-characterization}.
\end{proof}


\begin{remark}
\label{rem:thm1refinements}
The above proof shows that to obtain the conclusion of the
theorem, it suffices to impose the symmetry property
\eqref{eq:symmetry-general} on strings of the form
$\bfa=(\bfa',t)$, $t\in\NN$, 
where $\bfa'$ runs through a set of finite strings of positive integers
with the property that the rational numbers $[\bfa']$ represented by
these strings are dense in $[0,1]$.  
In fact, using set theoretic arguments one can show the following
refinement of the theorem: Given any infinite set $S$ of positive
integers, there exists a sequence of ``test strings'' $\bfa_s$, 
$s\in S$, where $\bfa_s$ has length
$s$, such that 
the conclusion of the theorem remains valid 
if the symmetry property holds for the strings $\bfa_s$, $s\in S$.
\end{remark}


\section{Proof of Theorem 
\protect \ref{thm:gk-characterization-optimality}}

\label{sec:thm2proof}

Given $N\in\NN$, we seek to construct a probability measure $\mu$ on
$[0,1]$ with continuous density function that satisfies the symmetry
property
\begin{equation}
\label{eq:proof1.2-symmetry-general}
\mu(I(\bfa))=\mu(I(\bfar))
\end{equation}
for all strings $\bfa$ of length at most
$N$, but is different from the Gauss--Kuzmin measure $\mu_{GK}$.

The key to our argument is contained in the following lemma.

\begin{lemma}
\label{lem:proof1.2-lem}
Let $N\in \NN$ be given and let $\mu$ be a probability measure
on $[0,1]$ satisfying 
\begin{equation}
\label{eq:proof1.2-lem1}
\mu(I(\bfa))=\mu_{GK}(I(\bfa))
\end{equation}
for all strings $\bfa=(a_1,\dots,a_N)$ of positive integers of length
exactly $N$. Then $\mu$ satisfies the symmetry property
\eqref{eq:proof1.2-symmetry-general} for all strings
$\bfa=(a_1,\dots,a_n)$ of length $n\le N$.
\end{lemma}

\begin{proof}
First note that, since the Gauss--Kuzmin measure $\mu_{GK}$ satisfies the
symmetry property \eqref{eq:proof1.2-symmetry-general},
the assumptions of the lemma imply that 
for strings $\bfa$ of length exactly $N$, 
\begin{equation}
\notag
\mu(I(\bfa)) =\mu_{GK}(I(\bfa)) =\mu_{GK}(I(\bfar))= \mu(I(\bfar)).
\end{equation}
Thus, $\mu$ satisfies the symmetry property
\eqref{eq:proof1.2-symmetry-general} for strings of length $N$.

It therefore remains to show that \eqref{eq:proof1.2-symmetry-general}
also holds for strings  $\bfa$ of length $n<N$. This will follow if we
can show that the assumption \eqref{eq:proof1.2-lem1} remains valid for
such strings.

To see this, let $\bfa=(a_1,\dots,a_n)$ be a string of positive integers
of length $n<N$. From the definition of $I(\bfa)$ as the set of real
numbers in $(0,1)$ whose continued fraction expansion begins with the
digits $a_1,\dots,a_n$ it is clear that, modulo a set of measure
zero\footnote{The exceptional set consists of rational numbers whose
continued fraction expansion begins with the string $\bfa$, but has
fewer than $N$ digits and thus is not counted in any set 
$I(\bfa')$, where $\bfa'\in\NN^N$.}, 
$I(\bfa)$ is the disjoint 
union of sets 
$I(\bfa')$, where $\bfa'$ runs through strings of the form
$\bfa'=(\bfa,\bfb)=(a_1,\dots,a_n,b_1,\dots,b_{N-n})$, with $b_i\in\NN$.
Since the strings $\bfa'$ have length exactly $N$, 
the assumption of the lemma applies to these strings, so we have 
\begin{align*}
\mu(I(\bfa))
&= \sum_{\bfb\in\NN^{N-n}} \mu(I(\bfa,\bfb))
= \sum_{\bfb\in\NN^{N-n}} \mu_{GK}(I(\bfa,\bfb))=\mu_{GK}(\bfa), 
\end{align*}
as claimed. 
\end{proof}

\begin{proof}[Proof of Theorem 
\ref{thm:gk-characterization-optimality}]
Let $N\in\NN$ be given.  In view of the lemma, it suffices to construct
a probability measure $\mu$ on $[0,1]$ with continuous density function
that takes on the same value as the Gauss--Kuzmin measure $\mu_{GK}$ on
sets of the form $I(\bfa)$, $\bfa\in\NN^N$, 
but is different from $\mu_{GK}$. 

Choose a
particular string $\bfa_0=(a_{1,0},\dots,a_{N,0})$ of $N$ positive integers,
let $\alpha<\beta$ be the endpoints of the interval
$I(\bfa_0)$, and let $\mu$ be the measure on $[0,1]$ with density function  
given by
\begin{equation}
\label{eq:proof1.2-f-def}
f(x)=\begin{cases}
f_0(x)&\text{if 
$\alpha<x<\beta$,}
\\
f_{GK}(x) &\text{otherwise,}
\end{cases}
\end{equation}
where 
\begin{equation}
\notag
f_{GK}(x)=\frac{1}{(\log 2)(1+x)}\quad (0\le x\le 1)
\end{equation}
is the density function of the Gauss--Kuzmin measure $\mu_{GK}$ and 
$f_0(x)$ is any nonnegative continuous function on
$[\alpha,\beta]$ satisfying 
\begin{align}
\label{eq:proof1.2-f0-property1}
f_0(\alpha)&=f_{GK}(\alpha),\quad f_0(\beta)=f_{GK}(\beta),
\\
\label{eq:proof1.2-f0-property2}
f_0(x)&\not=f_{GK}(x)\quad \text{for some $x\in(\alpha,\beta)$,}
\\
\label{eq:proof1.2-f0-property3}
\int_{\alpha}^{\beta}f_0(x)dx&=\int_{\alpha}^{\beta}f_{GK}(x)dx.
\end{align}
The definition \eqref{eq:proof1.2-f-def} of $f(x)$ implies that  
$f(x)$ is identical to $f_{GK}(x)$
outside the interval $I(\bfa_0)$. Since the intervals $I(\bfa)$,
$\bfa\in\NN^N$, are pairwise disjoint, it follows that
$\mu(I(\bfa))=\mu_{GK}(I(\bfa))$ holds for all strings $\bfa\in\NN^N$
with $\bfa\not=\bfa_0$, while condition \eqref{eq:proof1.2-f0-property3}
ensures that $\mu(I(\bfa))=\mu_{GK}(I(\bfa))$ also holds for $\bfa=\bfa_0$.
On the other hand, by  \eqref{eq:proof1.2-f0-property2} and the
continuity of $f_0$, the measure $\mu$ is different from the
Gauss--Kuzmin measure $\mu_{GK}$. 

The conditions \eqref{eq:proof1.2-f0-property1} and the
assumption that $f_0$ is continuous on $(\alpha,\beta)$ ensure that the
function $f(x)$ defined by \eqref{eq:proof1.2-f-def} is continuous on
the entire interval $[0,1]$.  Moreover, \eqref{eq:proof1.2-f0-property3}
implies that $\int_0^1f(x)dx=\int_0^1f_{GK}(x)dx=1$, so that $f(x)$ is a
continuous probability density function on $[0,1]$.  Thus, 
by Lemma \ref{lem:proof1.2-lem},
the measure
$\mu$ has all of the required properties, and the proof is complete.
\end{proof}


\section{Proof of Theorem 
\protect
\ref{thm:nontrivial-symmetries}}

\label{sec:thm3proof}

\subsection{Characteristic Numbers}

By Lemma \ref{lem:gk-formulas} the Gauss--Kuzmin measure $P_{GK}(\bfa)$
of a  string $\bfa=(a_1,\dots,a_n)$ of positive integers 
is given by
\begin{equation}
\label{eq:gk-formula1}
P_{GK}(\bfa)
= \left|\log_2\left(1+\frac{(-1)^n}{(p+q)(q'+q)}\right)\right|.
\end{equation}
where $p',q',p,q$ are the entries of the convergent
matrix 
$C(\bfa)=\left(\begin{smallmatrix} p'&p\\q'&q\end{smallmatrix}\right)$.
In particular, $P_{GK}(\bfa)$ depends only on the parity of $n$ and
the quantity $(p+q)(q'+q)$ appearing in the denominator on the right 
side of \eqref{eq:gk-formula1}.   In view of the key role played by this
quantity we introduce the following definition.

\begin{definition}[Characteristic number]
\label{def:char-number}
Let $\bfa=(a_1,\dots,a_n)$ be a finite string of positive integers with 
convergent matrix 
$C(\bfa)=\left(\begin{smallmatrix} p'&p\\q'&q\end{smallmatrix}\right)$.
The \emph{characteristic number}, $\chi(\bfa)$, 
of the string $\bfa$ is defined as  
\begin{equation}
\label{eq:char-number-def}
\chi(\bfa)=(p+q)(q'+q).
\end{equation}
Thus, $\chi(\bfa)$ is the product of the sum of the entries in
the second column with the sum of the entries of the second row of the
convergent matrix, $C(\bfa)$, associated with this string.
\end{definition}

In light of the above remarks we then have the following result.

\begin{lemma}
\label{lem:char-number}
If $\bfa$ and $\bfb$ are strings of positive integers whose lengths have
the same parity, then
$P_{GK}(\bfa)=P_{GK}(\bfb)$ holds if and only if
$\chi(\bfa)=\chi(\bfb)$.
\end{lemma}

In particular, \emph{two permutations of a string have the same
Gauss--Kuzmin measure if and only if they have the same characteristic
number.} The proofs of both parts of 
Theorem \ref{thm:nontrivial-symmetries} as well as
the experimental results presented in 
Section \ref{sec:numerical-data} are based on this crucial observation.


\subsection{Proof of Theorem \ref{thm:nontrivial-symmetries}(i)}

Given a string $(a,b,c)$ of positive integers, let $\chi(a,b,c)$ be
the characteristic number of this string. We seek to show that if  
$(a',b',c')$ is a permutation of $(a,b,c)$ such that 
$(a',b',c')\not=(a,b,c)$ and $(a',b',c')\not=(c,b,a)$,
then $\chi(a',b',c')\not=\chi(a,b,c)$.  We first consider the particular
permutation $(a',b',c')=(b,a,c)$, i.e., the permutation that
interchanges $a$ with $b$.

\begin{lemma}
\label{lem:proof1.3-lem1}
Let $(a,b,c)$ be a string of positive integers.
If $a\not=b$, then 
\begin{equation}
\notag
\chi(a,b,c)\not=\chi(b,a,c).
\end{equation}
\end{lemma}
\begin{proof}
The proof is based on an explicit calculation of the characteristic
number of a string $(a,b,c)$ of positive integers as a polynomial 
in the variables $a,b,c$.

Fix a string $(a,b,c)$ of positive integers with $a\not=b$, and suppose,
to get a contradiction, that 
\begin{equation}
\label{eq:proof1.3-assumption-chi}
\chi(b,a,c)= \chi(a,b,c).
\end{equation}
Using \eqref{eq:key1}, we can calculate 
the convergent matrix of the string $(a,b,c)$ as 
\begin{equation*}
C(a,b,c)=\mat{a}\mat{b}\mat{c}
=\begin{pmatrix} b & bc+1\\ ab+1 & abc+a+c\end{pmatrix},
\end{equation*}
so we have $p'=b$, $q'=ab+1$, $p=bc+1$, and $q=abc+a+c$. Substituting these
values into \eqref{eq:char-number-def}, we obtain
\begin{align}
\label{eq:char-number-formula0}
\chi(a,b,c)&=(p+q)(q'+q)
=(abc+bc+a+c+1)(abc+ab+a+c+1).
\end{align}
Introducing the polynomial 
\begin{equation}
\label{eq:proof1.3-S-def}
S=S(a,b,c)=abc+a+b+c+1,
\end{equation}
we can rewrite \eqref{eq:char-number-formula0} as 
\begin{equation}
\label{eq:char-number-formula1}
\chi(a,b,c)=(S+bc-b)(S+ab-b)=S^2 +S(ab+bc-2b)+b^2(a-1)(c-1).
\end{equation}

Now note that, since $S$ is a symmetric polynomial in $a,b,c$,
permuting these variables does not affect the value of $S$. 
Hence, interchanging $a$ and $b$ in \eqref{eq:char-number-formula1}
and subtracting the resulting expression from the expression
on the right of \eqref{eq:char-number-formula1} yields, 
in view of our assumption \eqref{eq:proof1.3-assumption-chi},
\begin{align}
\notag
0&=\chi(a,b,c)-\chi(b,a,c)
\\
\notag
&=(bc-2b-ac+2a)S+(b^2(a-1)-a^2(b-1))(c-1)
\\
\notag
&=(b-a)
\left((c-2)S+(ab-a-b)(c-1)\right),
\end{align}
and hence, since $a\not=b$, 
\begin{equation}
\label{eq:char-number-formula3}
(c-2)S=(c-1)(a+b-ab)=(c-1)\bigl[1-(a-1)(b-1)\bigr].
\end{equation}
We show that \eqref{eq:char-number-formula3} cannot hold. 
If $c=1$, the right-hand side of \eqref{eq:char-number-formula3}
vanishes, while the left-hand side is negative, so we have a
contradiction. 
If $c=2$, the left-hand side vanishes, while the right-hand side is
non-zero since $(a-1)(b-1)\not=1$ by our assumption $a\not=b$, so this
case is also impossible.
Finally, if $c\ge 3$, then \eqref{eq:char-number-formula3} implies
\begin{equation*}
S=\frac{c-1}{c-2}\bigl[1-(a-1)(b-1)\bigr]\le \frac{c-1}{c-2}\le 2,
\end{equation*}
which is again a contradiction since, by  
\eqref{eq:proof1.3-S-def}, $S=abc+a+b+c+1\ge 5$.
\end{proof}

\begin{proof}[Proof of Theorem \ref{thm:nontrivial-symmetries}(i)]
Let $(a,b,c)$ be a string of positive integers, and let 
$(a',b',c')$ be a permutation of $(a,b,c)$. Suppose that 
\begin{equation}
\label{eq:proof1.3-assumption}
\chi(a',b',c') = \chi(a,b,c).
\end{equation}
We seek to show that \eqref{eq:proof1.3-assumption}
can only hold if the permutation $(a',b',c')$
is either the string $(a,b,c)$ itself, or its reverse, $(c,b,a)$.

If $b'=b$, the desired conclusion obviously holds. Assume therefore that 
$b'\not=b$. Then $b'=a$ or $b'=c$, and by the symmetry property we may 
assume without loss of generality that $b'=a$. Thus, we have 
either $(a',b',c')=(b,a,c)$ or $(a',b',c')=(c,a,b)$. Since, by the
symmetry property, $\chi(b,a,c)=\chi(c,a,b)$, it suffices to consider
the case $(a',b',c')=(b,a,c)$. But in this case 
Lemma \ref{lem:proof1.3-lem1} along with 
our assumption \eqref{eq:proof1.3-assumption} implies that  $a=b$. 
Hence, $(a',b',c')=(a,a,c)=(a,b,c)$, i.e., the permutation $(a',b',c')$ 
is the identity permutation. This completes the proof.
\end{proof}


\subsection{Proof of Theorem \ref{thm:nontrivial-symmetries}(ii)}

For part (ii) of Theorem \ref{thm:nontrivial-symmetries} we seek to
construct, for any given length $n\ge4$, an 
$\lfloor (n-2)/2\rfloor$-parameter family of strings
$\bfa$ of length $n$ that have a nontrivial symmetry in the sense of
Definition \ref{def:nontrivial-symmetries}.  In view of Lemma
\ref{lem:char-number}, this amounts to constructing strings $\bfa$
for which there exists a permutation $\sigma(\bfa)$ with $\sigma(\bfa)\not=\bfa$
and $\sigma(\bfa)\not=\bfar$ that has the same characteristic number as
the string $\bfa$.

The key to our construction lies in a special class of strings defined as follows.

\begin{definition}[Stable strings]
\label{def:stable-string}
Let $\bfa=(a_1,\dots,a_n)$ be a finite string of positive integers and
let 
$C(\bfa)=\left(\begin{smallmatrix} p'&p\\q'&q\end{smallmatrix}\right)$
be the convergent matrix of $\bfa$.
The string $\bfa$ is called \emph{stable} if it satisfies 
\begin{equation}
\label{eq:def-stable}
p=2q',
\end{equation}
In other words, a stable string is a string whose convergent matrix has the
property that its $(1,2)$ entry is exactly twice its $(2,1)$ entry.
\end{definition}


\begin{lemma}
\label{lem:stable-families}
\mbox{}
\begin{itemize}
\item[(i)] The following families of strings are stable:
\begin{align}
\label{eq:length2-stable-strings}
&(t,2t)
\quad (t\in\NN),
\\
\label{eq:length3-stable-strings}
&(t,1,2t+1)
\quad (t\in\NN).
\end{align}
\item[(ii)]
If $\bfa=(a_1,\dots,a_n)$ is a stable string, then any string of the
form
\begin{align}
\notag
&(t,\bfar, 2t)=(t,a_n,\dots,a_1,2t)
\quad (t\in\NN)
\end{align}
is also stable.
\item[(iii)]
For each $n\ge2$ there exists an $\lfloor n/2\rfloor$-parameter
family of stable strings of length $n$.
\end{itemize}
\end{lemma}

\begin{proof}
(i) Using \eqref{eq:key1}, we calculate 
the convergent matrices associated with the strings
\eqref{eq:length2-stable-strings} and
\eqref{eq:length3-stable-strings}:
\begin{align*}
C(t,2t)&=\mat{t}\mat{2t}=\begin{pmatrix} 1&2t\\ t & 1 + 2t^2\end{pmatrix},
\\
C(t,1,2t+1)&=\mat{t}\mat{1}\mat{2t+1}
=\begin{pmatrix} 1&2+2t\\ 1+t & 1 + 4t+2t^2\end{pmatrix}.
\end{align*}
Both of these matrices satisfy the stability condition  \eqref{eq:def-stable}
for any $t\in\NN$, so the associated families of strings are stable as
claimed. 

\medskip
(ii) 
Assume $\bfa=(a_1,\dots,a_n)$ is a stable string with convergent matrix
$C(\bfa)=\left(\begin{smallmatrix} p'&p\\q'&q\end{smallmatrix}\right)$.
Using the relation (cf. Lemma \ref{lem:convergent-matrices}(i))
$C(\bfar)= \left(\begin{smallmatrix}
p'&q'\\p&q\end{smallmatrix}\right)$,
we obtain, for any $t\in\NN$,
\begin{align*}
C(t,\bfar,2t)&=\mat{t}C(\bfar)\mat{2t}
=\mat{t}\begin{pmatrix} p'&q'\\p&q\end{pmatrix}\mat{2t}
\\
&=\begin{pmatrix}q & p+2tq \\ q'+tq & p'+2tq' + tp +2t^2q\end{pmatrix}
=\begin{pmatrix} r'&r\\s'&s\end{pmatrix},
\end{align*}
say.   Since the string $\bfa$ is stable, we have $p=2q'$. It follows that 
$r=p+2tq=2q'+2tq=2(q'+tq)=2s'$,
so the string $(t,\bfar,2t)$ is stable as well.

\medskip
(iii) This follows from parts (i) and (ii) 
by starting out with the families of strings 
\eqref{eq:length2-stable-strings} and
\eqref{eq:length3-stable-strings} and inductively applying the
procedure described in part (ii) of the lemma. At each step the length of
the string is increased by $2$ and an additional free parameter is
introduced.  Thus the total number of free parameters in the families
of strings generated by this process is $\lfloor n/2\rfloor$, where $n$
is the length of the string. 
\end{proof}

\begin{remark}
\label{rem:explicit-stable-families}
The families of stable strings generated by the iterative procedure of
Lemma \ref{lem:stable-families} can be described explicitly. In the case
$n=2m$ is even, the strings are of the form
\begin{equation}
\notag
(\delta_1 t_m,\delta_2 t_{m-1},\dots,\delta_m t_1,
\delta_{m-1} t_1,\dots,\delta_{1} t_{m-1}, \delta_{0} t_m)
\quad (t_1,\dots,t_m\in\NN),
\end{equation}
where $\delta_i=1$ if $i$ is odd and  $\delta_i=2$ if $i$ is even.
A similar, though slightly more complicated,
explicit description could be given for the case when $n$ is odd.
Table \ref{table:stable-families} shows the families of strings obtained
from the lemma for lengths $n\le 7$.

\begin{table}[H]
\begin{center}
\addtolength{\tabcolsep}{6pt}
\renewcommand{\arraystretch}{1.5}
\begin{tabular}{|c|c|}
\hline
$n$& $\bfa$ 
\\
\hline
$2$ & $(t_1,2t_1)$ 
\\
$3$ & $(t_1,1, 2t_1+1)$
\\
$4$ & 
$(t_2,2t_1, t_1, 2t_2)$  
\\
$5$ &
$(t_2, 2t_1+1,1,t_1,2t_2)$
\\
$6$ &
$(t_3,2t_2,t_1,2t_1,t_2,2t_3)$
\\
$7$ &
$(t_3, 2t_2, t_1, 1, 2t_1+1, t_2, 2t_3)$
\\
\hline
\end{tabular}
\caption{Families of stable strings $\bfa$ of length
$n\in\{2,3,4,5,6,7\}$ constructed by the procedure
of Lemma \ref{lem:stable-families}. Here $t_1,t_2,\dots$ are arbitrary
positive integer parameters.} 
\label{table:stable-families} 
\end{center}
\end{table}
\end{remark}


We next show 
that any stable string of length $n$
yields a string of length $n+2$ that has a nontrivial symmetry.
Given a string $\bfa=(a_1,\dots,a_n)$ of positive integers, let $\bfap$
be the string defined by 
\begin{equation}
\label{eq:def-bfap}
\bfap=(2,1,a_1,\dots,a_{n-1},a_n+1).
\end{equation}
Thus, $\bfap$ is the string obtained from $\bfa$ by prepending the
digits $2$ and $1$ and incrementing the last digit in $\bfa$ by $1$.

\begin{lemma}
\label{lem:stable-to-nontrivial-symmetries}
If $\bfa=(a_1,\dots,a_n)$ is a stable string of positive integers, then
the string $\bfap$ defined by \eqref{eq:def-bfap}
has a nontrivial symmetry $\sigma$ given by 
\begin{equation}
\label{eq:proof1.4-nontrivial-symmetry}
\sigma(\bfap)=(2,a_{n}+1,a_{n-1},\dots,a_1,1).
\end{equation}
\end{lemma}

\begin{proof}
Note that $\sigma(\bfap)$ is the permutation that reverses the last
$n+1$ digits of $\bfap$. Since $a_n+1\not=1$, 
this permutation cannot be the identity permutation, and since
$2\not=1$, it cannot be the permutation that reverses the digits of
$\bfap$.  Thus, $\sigma$ is a nontrivial permutation in the sense of
Definition \ref{def:nontrivial-symmetries}, and in view of 
Lemma \ref{lem:char-number}, it therefore remains to show that the strings
$\bfap$ and $\sigma(\bfap)$ have the same characteristic number, i.e., 
that
\begin{equation}
\label{eq:proof1.4-char-number-equality}
\chi(\bfap)=\chi(\sigma(\bfap)).
\end{equation}
Using 
the identity
$\left(\begin{smallmatrix} 0&1\\1&a+1\end{smallmatrix}\right)=
\left(\begin{smallmatrix} 0&1\\1&a\end{smallmatrix}\right)
\left(\begin{smallmatrix} 1&1\\0&1\end{smallmatrix}\right)$
and \eqref{eq:key1} 
we obtain 
\begin{align}
\notag
C(\bfap)&
=\mat{2}\mat{1}\mat{a_1}\dots\mat{a_{n-1}}\mat{a_n+1}
\\
\notag
&=\mat{2}\mat{1}C(\bfa)\begin{pmatrix} 1&1 \\ 0 & 1\end{pmatrix}
=\mat{2}\mat{1}
\begin{pmatrix} p'&p\\q'&q\end{pmatrix}
\begin{pmatrix} 1&1 \\ 0 & 1\end{pmatrix}
\\
\notag
&=\begin{pmatrix} 
p'+q' & p'+q'+p+q\\
2p'+3q' & 2p'+3q'+2p+ 3q
\end{pmatrix} 
= \begin{pmatrix} r'&r\\s'&s\end{pmatrix},
\end{align}
say.
It follows that 
\begin{align}
\label{eq:chibfap-calculation}
\chi(\bfap)&=(r+s)(s'+s)
=(3p'+4q'+3p+4q)(4p'+6q'+2p+3q)
\\
\notag
&=(3p'+10q'+4q)(4p'+10q'+3q),
\end{align}
where the last step follows from the assumption that the string $\bfa$ is
stable and thus satisfies $p=2q'$.

Similarly, noting that $\sigma(\bfap)$ is the string 
obtained by reversing the last $n+1$ digits of $\bfap$ (i.e., all digits 
of $\bfap$ after the leading digit $2$), we obtain 
\begin{align}
\notag
C(\sigma(\bfap))
&=\mat{2}\left[
\mat{1}
\begin{pmatrix} p'&p \\ q' & q\end{pmatrix}
\begin{pmatrix} 1&1 \\ 0 & 1\end{pmatrix}
\right]^T
\\
\notag
&=\mat{2}
\begin{pmatrix} 1&0\\1&1\end{pmatrix}
\begin{pmatrix} p'&q' \\ p & q\end{pmatrix}
\mat{1}
\\
\notag
&=\begin{pmatrix}
q'+q & p'+q'+p+q \\
3q' + 2q & 3p'+3q'+2p+2q
\end{pmatrix}
= \begin{pmatrix} u'&u\\v'&v\end{pmatrix},
\end{align}
say, and hence 
\begin{align}
\notag
\chi(\sigma(\bfap))
&=(u+v)(v'+v)=(4p'+4q'+3p+3q)(3p'+6q'+2p+4q)
\\
\label{eq:chisigmabfap-calculation}
&=(4p'+10q'+3q)(3p'+10q'+4q).
\end{align}
Comparing \eqref{eq:chibfap-calculation} 
with \eqref{eq:chisigmabfap-calculation} yields the desired symmetry
relation \eqref{eq:proof1.4-char-number-equality}.
\end{proof}

Table \ref{table:families-with-nontrivial-symmetries}
illustrates the construction of strings with nontrivial symmetries from
stable strings  described in Lemmas \ref{lem:stable-families} and
\ref{lem:stable-to-nontrivial-symmetries}. 


\begin{table}[H]
\begin{center}
\addtolength{\tabcolsep}{-3pt}
\renewcommand{\arraystretch}{1.5}
\begin{tabular}{|c|c|c|}
\hline
$\bfa$ & $\bfap$ & $\sigma(\bfap)$
\\
\hline
$(t_1,2t_1)$ & $(2,1,t_1,2t_1+1)$ & $(2,2t_1+1,t_1,1)$ 
\\
$(t_1,1, 2t_1+1)$ & $(2,1,t_1,1,2t_1+2)$ & $(2,2t_1+2,1,t_1, 1)$ 
\\
$(t_2,2t_1, t_1, 2t_2)$ & 
$(2,1,t_2,2t_1,t_1,2t_2+1)$ &
$(2,2t_2+1, t_1, 2t_1, t_2, 1)$
\\
$(t_2,2t_1+1,1,t_1,2t_2)$ &
$(2,1,t_2,2t_1+1,1,t_1,2t_2+1)$ &
$(2,2t_2+1,t_1, 1, 2t_1+1, t_2, 1)$ 
\\
\hline
\end{tabular}
\caption{Families of stable strings $\bfa$ of length $n\in\{2,3,4,5\}$ and the
associated families of strings $\bfap$ of length $n+2$ 
with nontrivial symmetries
$\sigma(\bfap)$ given by \eqref{eq:def-bfap}
and \eqref{eq:proof1.4-nontrivial-symmetry}
Here $t_1$ and $t_2$ are
arbitrary positive integer parameters.}
\label{table:families-with-nontrivial-symmetries}. 
\end{center}
\end{table}

\begin{proof}[Proof of Theorem \ref{thm:nontrivial-symmetries}(ii)]
Let $n\ge 4$ be given and set $n'=n-2$ (so that $n'\ge 2$). By Lemma
\ref{lem:stable-families}(iii) there exists an $\lfloor n'/2\rfloor$-parameter
family of stable strings $\bfa$ of length $n'$. By Lemma
\ref{lem:stable-to-nontrivial-symmetries} each such string $\bfa$ gives
rise to a string $\bfap$ of length $n'+2=n$ that has a nontrivial
symmetry, and the mapping $\bfa\to \bfap$ is obviously injective. Thus, the
strings $\bfap$ obtained in this manner form 
an $\lfloor (n-2)/2\rfloor$-parameter family of strings of length $n$ 
with a nontrivial symmetry, as claimed.
\end{proof}

\begin{remark}
\label{rem:generalized-stable-families}
The families of strings $\bfap$ of even lengths
constructed via Lemmas \ref{lem:stable-families} and
\ref{lem:stable-to-nontrivial-symmetries} 
involve each of the digits $2$
and $1$ exactly once, along with digits of the forms $(*)$ $t_i$,
$2t_i$, $2t_i+1$ 
(cf.  Table \ref{table:families-with-nontrivial-symmetries}). 
By requiring the parameters $t_i$
to be pairwise distinct and satisfy $t_i\equiv 1\bmod 4$ and $t_i>2$
one can ensure that the digits of the strings $\bfap$ of even length
obtained from this construction are pairwise distinct.  
For strings $\bfap$ of odd length $5$ or greater, 
this is not the case, as the construction given above necessarily
involves two
occurrences of the digit $1$ (see the cases $n=3$ and $n=5$ in Table
\ref{table:families-with-nontrivial-symmetries}), although by
restricting the parameters $t_i$ as before one can ensure that all
remaining digits of $\bfap$ are  pairwise distinct. 
The duplication 
of the digit $1$ in the case of strings of odd length can be avoided 
by a generalized version of the construction given in Lemmas
\ref{lem:stable-families} and \ref{lem:stable-to-nontrivial-symmetries}
that involves an additional parameter $s$.  We give a brief sketch of the
argument.
Define a string $\bfa$ to be \emph{$s$-stable} if it satisfies
$p=(s^2+s)q'$. The latter condition generalizes the stability condition
\eqref{eq:def-stable}, which corresponds to the case $s=1$. 
Similar to 
Lemmas \ref{lem:stable-families} and \ref{lem:stable-to-nontrivial-symmetries}
one can verify that all strings of the form $(t,(s^2+s)t)$ 
and $(t,s^2+s-1, (s^2+s)t+1)$, where $t\in\NN$, 
are $s$-stable, and then use an inductive process 
to construct, for each $n\ge2$, an infinite $\lfloor
n/2\rfloor$-parameter family of $s$-stable strings, each of which gives
rise to a string $\bfap$  with a nontrivial symmetry.
For odd $n\ge5$ the strings $\bfap$  of length $n$ 
obtained in this manner involve $\lfloor (n-2)/n\rfloor$ additional parameters
$t_i$ and digits of the form $2$, $s$, $2s$, 
$t_i$, $(s^2+s)t_i$, and $(s^2+s)t_i+s$. 
These digits will be pairwise distinct if we  
require the parameter $s$ to be greater than $2$ and the parameters $t_i$ to be
pairwise distinct and congruent to $1$ modulo $2s^2+2s$.
\end{remark}


\section{Numerical Data and Conjectures}
\label{sec:numerical-data}

Recall (cf. Definition \ref{def:nontrivial-symmetries}) that a string
$\bfa=(a_1,\dots,a_n)$ of positive integers is said to have a \emph{nontrivial
symmetry} if there exists a permutation of $\bfa$ other than the identity
and the reverse that preserves the Gauss--Kuzmin measure $P_{GK}(\bfa)$
of the string.  By Theorem \ref{thm:nontrivial-symmetries}, strings of
length $n=3$ have no nontrivial symmetries,  while for each $n\ge 4$
there exists an infinite family of strings of length $n$ that \emph{do}
have a nontrivial symmetry.  Conjecture \ref{conj:nontrivial-symmetries}
states that strings of the latter type are the exception rather than the
rule in the sense that their proportion among all strings of length $n$
of digits in $\{1,\dots,N\}$ tends to $0$ as $N\to\infty$.  In this
section, we provide numerical evidence supporting this conjecture, and
we propose refined versions of this conjecture.

For simplicity, we consider only strings of distinct
digits. This restriction does not affect the assertion of
Conjecture \ref{conj:nontrivial-symmetries} since, for each fixed $n$,
the proportion of strings $\bfa=(a_1,\dots,a_n)\in\{1,\dots,N\}^n$ that
have distinct digits is $N(N-1)\dots(N-n+1)/N^n$, which converges to $1$ 
as $N\to\infty$. 

Given an $n$-tuple $(a_1,\dots,a_n)$ 
of distinct positive integers, let 
\begin{equation}
\label{eq:nubfa-def}
\nu(\bfa)=\#\{P_{GK}(\sigma(\bfa)): \sigma\in S_n \},
\end{equation}
where $S_n$ is the set of all permutations on $\{1,\dots,n\}$.
Thus, $\nu(\bfa)$ is the number of distinct values of the Gauss--Kuzmin 
measure $P_{GK}(\sigma(\bfa))$ as $\sigma(\bfa)$ runs through the $n!$
permutations of $\{a_1,\dots,a_n\}$.   Note that $\nu(\bfa)$
depends only on the digits $a_1,\dots,a_n$, not on the order in which
these digits occur in the string $\bfa$.  We may therefore assume that
$a_1<\dots<a_n$.

Trivially, we have $\nu(\bfa)\le n!$. Moreover, pairing up each
permutation of $\bfa$ with its reverse, we see that the symmetry property
\eqref{eq:symmetry-alt}
implies  $\nu(\bfa)\le n!/2$, 
with equality if and only if none 
of the permutations of $\bfa$ has a nontrivial symmetry.  Thus, a
natural way to quantify the occurrence of strings with nontrivial
symmetries is to compare, for large $N\in\NN$,
the number of tuples $\bfa=(a_1,\dots,a_n)\in\NN^n$ with 
$a_1<\dots<a_n\le N$ that satisfy the condition $\nu(\bfa)<n!/2$
with the total number of such tuples, i.e., with $\binom{N}{n}$. 
Set
\begin{align}
\notag
f(N,n)&=\#\left\{
\bfa=(a_1,\dots,a_n)\in\NN^n: a_1<\dots<a_n\le N, 
\ \nu(\bfa)<\frac{n!}{2} \right\},
\\
\notag
\delta(N,n)&=\frac{f(N,n)}{\binom{N}{n}}.
\end{align}
The quantity $\delta(N,n)$ represents the probability
that a random sample of $n$ distinct digits in $\{1,\dots,N\}$ 
has a permutation with a nontrivial symmetry.

The quantities $\nu(\bfa)$ defined in \eqref{eq:nubfa-def}, and hence
the numbers $f(N,n)$ and $\delta(N,n)$, 
can be computed from the explicit formula
\eqref{eq:gk-formula} for the Gauss--Kuzmin measure $P_{GK}(\bfa)$.
Using the symbolic computation software \emph{Mathematica}, we
carried out these computations for $n\in\{4,5,6\}$ and a range of values
of $N$.  

For $n=4$, we computed the exact values of $f(N,4)$ for all positive
integers $N\le 120$.  Table \ref{table:fN4} and Figure \ref{fig:fN4} 
below show the results of these computations.  The table lists, for
$N=10,20,\dots,120$, the total number, $\binom{N}{4}$, of unordered
$4$-tuples of distinct positive integers at most $N$, along with the number,
$f(N,4)$, and proportion, $\delta(N,4)=f(N,4)/\binom{N}{4}$, of these
tuples that have a permutation with a nontrivial symmetry.  Also shown
is the ratio $f(N,4)/N$, which measures the rate of growth of the
function $f(N,4)$ compared to that of a linear function in $N$.

\begin{table}[H]
\begin{center}
\addtolength{\tabcolsep}{6pt}
\renewcommand{\arraystretch}{1.5}
\begin{tabular}{|c|r|r|r|r|}
\hline
$N$ & $\binom{N}{4}$ & $f(N,4)$ & $f(N,4)/N$ & $\delta(N,4)$ 
\\
\hline
10 & 210 & 10 & 1.0000 & 0.047619 \\
 20 & 4845 & 30 & 1.5000 & 0.006192 \\
 30 & 27405 & 47 & 1.5667 & 0.001715 \\
 40 & 91390 & 66 & 1.6500 & 0.000722 \\
 50 & 230300 & 87 & 1.7400 & 0.000378 \\
 60 & 487635 & 104 & 1.7333 & 0.000213 \\
 70 & 916895 & 121 & 1.7286 & 0.000132 \\
 80 & 1581580 & 142 & 1.7750 & 0.000090 \\
 90 & 2555190 & 159 & 1.7667 & 0.000062 \\
 100 & 3921225 & 178 & 1.7800 & 0.000045 \\
 110 & 5773185 & 199 & 1.8091 & 0.000034 \\
 120 & 8214570 & 216 & 1.8000 & 0.000026 \\
 \hline
\end{tabular}
\caption{Strings of length $4$ with nontrivial symmetries.}
\label{table:fN4}
\end{center}
\end{table}


\begin{figure}[H]
\begin{center}
\includegraphics[width=0.7\textwidth]{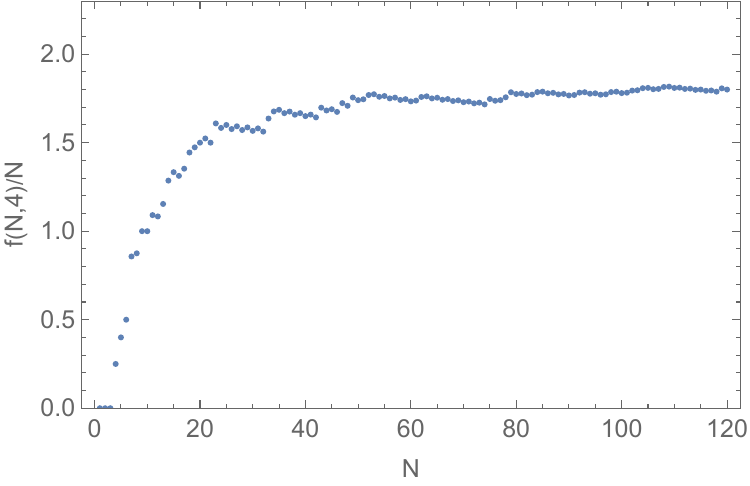}
\caption{The ratios $f(N,4)/N$ for $N\le 120$.}
\label{fig:fN4}
\end{center}
\end{figure}

The data shown in Table \ref{table:fN4}
and Figure \ref{fig:fN4}
provide rather compelling evidence that strings of
length $4$ with nontrivial symmetries are exceedingly rare.  
For example, 
among the $\binom{120}{4}=8,214,570$ unordered $4$-tuples of distinct
integers in $\{1,\dots,120\}$, only $216$ have a permutation with a
nontrivial symmetry. 
The probabilities
$\delta(N,4)$ listed in the last column of the table
seem to converge rapidly to $0$ as $N\to\infty$.
Moreover, the counts $f(N,4)$ of strings with nontrivial symmetries 
appear to grow at a rate that is roughly linear in $N$.
We are therefore led to the following conjecture.

\begin{conjecture}
\label{conj:fN4}
We have:
\begin{itemize}
\item[(i)]
$\displaystyle \lim_{N\to\infty} \delta(N,4)=0$;
\item[(ii)]
$\displaystyle \lim_{N\to\infty}\frac{\log f(N,4)}{\log N}=1$.
\end{itemize}
\end{conjecture}

Since $\delta(N,4)=f(N,4)/\binom{N}{4}=O(f(N,4)/N^4)$, 
part (ii) of the conjecture is a stronger version of part (i), implying
that $\delta(N,4)$ converges to $0$ at a rate $O(1/N^c)$ for any
constant $c<3$.
We note that the construction given 
in the proof of Theorem \ref{thm:nontrivial-symmetries}(ii)
implies $f(N,4)\gg N$ and hence yields 
$\liminf_{N\to\infty}\log f(N,4)/\log N=1$, i.e., the lower bound 
in part (ii) of the conjecture.

Tables \ref{table:fN5} and \ref{table:fN6} below show the results
of analogous computations for strings of length $n\in\{5,6\}$.  Because
of running time limitations\footnote{The number of cases to be examined
grows at a rate proportional to $N^n$.},
the range for the variable $N$ had to be
restricted to $N\le60$ for $n=5$ and $N\le 35$ for $n=6$.

\begin{table}[H]
\begin{center}
\addtolength{\tabcolsep}{6pt}
\renewcommand{\arraystretch}{1.5}
\begin{tabular}{|c|r|r|r|}
\hline
$N$ & $\binom{N}{5}$ & $f(N,5)$ & $\delta(N,5)$ 
\\
\hline
 10 & 252 & 8 & 0.031746 \\
 20 & 15504 & 43 & 0.002773 \\
 30 & 142506 & 85 & 0.000596 \\
 40 & 658008 & 137 & 0.000208 \\
 50 & 2118760 & 184 & 0.000087 \\
 60 & 5461512 & 236 & 0.000043 \\
 \hline
\end{tabular}
\caption{Strings of length $5$ with nontrivial symmetries.
}
\label{table:fN5}
\end{center}
\end{table}

\begin{table}[H]
\begin{center}
\addtolength{\tabcolsep}{6pt}
\renewcommand{\arraystretch}{1.5}
\begin{tabular}{|c|r|r|r|}
\hline
$N$ & $\binom{N}{6}$ & $f(N,6)$ & $\delta(N,6)$ 
\\
\hline
 10 & 210 & 23 & 0.109524 \\
 15 & 5005 & 100 & 0.019980 \\
 20 & 38760 & 276 & 0.007121 \\
 25 & 177100 & 496 & 0.002801 \\
 30 & 593775 & 746 & 0.001256 \\
 35 & 1623160 & 1088 & 0.000670 \\
\hline
\end{tabular}
\caption{Strings of length $6$ with nontrivial symmetries.
}
\label{table:fN6}
\end{center}
\end{table}

Again, the tables provide convincing evidence that the probabilities 
$\delta(N,n)$ approach $0$ as $N\to\infty$.
We make the following
conjecture, which generalizes part (i) of the latter conjecture to
strings of arbitrary length $n\ge4$, and implies Conjecture 
\ref{conj:nontrivial-symmetries}.

\begin{conjecture}
\label{conj:deltaNn}
For any integer $n\ge 4$ we have
\begin{equation}
\notag
\lim_{N\to\infty} \delta(N,n)=0.
\end{equation}
\end{conjecture}

Regarding the rate of growth of $f(N,n)$ for general $n\ge4$, it seems
plausible that, in analogy to part (ii) of  
Conjecture \ref{conj:fN4}, a relation of the form 
\begin{equation}
\label{eq:fNn-asymptotic}
\lim_{N\to\infty}\frac{\log f(N,n)}{\log N}=\alpha_n,
\end{equation}
or possibly even 
\begin{equation}
\label{eq:fNn-strong-asymptotic}
\lim_{N\to\infty}\frac{f(N,n)}{N^{\alpha_n}}=c_n,
\end{equation}
holds with suitable constants  $\alpha_n$ and $c_n>0$. Unfortunately, we
do not have enough data to support a specific conjecture of this type.
We note that the families of strings $\bfa$
constructed in the proof of Theorem 
\ref{thm:nontrivial-symmetries}(ii) involve 
$\lfloor(n-2)/2\rfloor$ parameters, and since the digits in $\bfa$ depend
linearly on these parameters, $f(N,n)$ must grow at a rate at least 
$N^{\lfloor(n-2)/2\rfloor}$ as $N\to \infty$. Thus, the constant
$\alpha_n$ in \eqref{eq:fNn-asymptotic}
and \eqref{eq:fNn-strong-asymptotic}
must satisfy
\begin{equation}
\label{eq:alphan-lower-bound}
\alpha_n\ge \left\lfloor\frac{n-2}{2}\right\rfloor.
\end{equation}
For $n=4$, the bound \eqref{eq:alphan-lower-bound} becomes $\alpha_4\ge
1$, and the data in Table \ref{table:fN4} suggests that this bound is
sharp, i.e., that $\alpha_4=1$.  Whether this remains true for $n>4$ is
an open question.

\section{Concluding Remarks}
\label{sec:concluding-remarks}

While the digits in decimal (and base $b$) expansions of real numbers
behave essentially like independent identically distributed random
variables, the statistical behavior of digits in continued fraction
expansions of real numbers is more complicated and far less intuitive.
The frequency with which a string $\bfa$ of digits occurs in the
continued fraction expansion of a random real number is given by
the Gauss--Kuzmin measure $P_{GK}(\bfa)$ of this string.  Classically,
this measure arises either as a solution to a functional equation (see,
e.g., \cite{khinchin-book}) or as the invariant measure with
respect to the ergodic transformation corresponding to the continued
fraction algorithm (see, e.g., \cite{dajani-kraaikamp-book}). 
In this paper we provided a new, elementary,
derivation of this measure by showing in Theorem
\ref{thm:gk-characterization} that the Gauss--Kuzmin measure is the
\emph{only}
(continuous) measure under which the reverse of a finite string of
continued fraction digits occurs in the continued fraction expansion of
a random real number with the same frequency as the original string.

Motivated by this result, we investigated more generally the extent to
which the frequency of a string of digits in the continued fraction
expansion depends on the order in which these digits appear in the
string. Specifically, we considered the question of whether the reverse of
a string is the \emph{only} permutation  under which this
frequency is invariant.  In Theorem \ref{thm:nontrivial-symmetries} we 
proved that this indeed holds for \emph{all} strings of length $3$,
while for each length $n\ge 4$ there exists an infinite family of
strings of length $n$ that do have a permutation other
than the reversal under which the frequency remains invariant.  
Supported by experimental data, we conjecture (Conjecture
\ref{conj:nontrivial-symmetries}) that strings of the latter type are 
the exception in the sense that for a ``typical'' string $\bfa$ of
continued fraction digits the reversal of the string is the \emph{only}
permutation that leaves this frequency invariant. 

We conclude this paper by mentioning some open problems suggested by
these results. The main---and arguably most interesting---set of open
problems concern the frequencies of strings with nontrivial symmetries.
Call such a string \emph{exceptional}.
One way to quantify the
occurrence of exceptional strings is via the function $\delta(N,n)$
defined in Section \ref{sec:numerical-data}, which can be 
interpreted as the probability that a randomly chosen string of $n$
distinct digits in $\{1,\dots,N\}$ has a permutation that is
exceptional.\footnote{%
Note that, in contrast to the quantity $\nu(\bfa)$ defined in
\eqref{eq:nubfa-def}, which depends 
only on the \emph{set} (or multi-set) of digits 
in a string $\bfa$, 
whether or not $\bfa$ is exceptional (i.e., has
a nontrivial symmetry) depends on the order of the digits in $\bfa$ and
thus is not invariant with respect to permuting these digits.}
What can one say about the asymptotic behavior of
$\delta(N,n)$ as $N\to\infty$?  

Alternatively, one can consider, for a given string
$\bfa=(a_1,\dots,a_n)$ of $n$ distinct positive integers, the quantity
$\epsilon(\bfa)=1-\nu(\bfa)/(n!/2)$, where $\nu(\bfa)$ is defined in
\eqref{eq:nubfa-def}.  As noted in Section \ref{sec:numerical-data}, we
have $\nu(\bfa)\le n!/2$, and hence $\epsilon(\bfa)\ge 0$, with
equality if and only if none of the permutations of the string $\bfa$ is
exceptional.  Thus, $\epsilon(\bfa)$ can be viewed as a measure for the frequency of 
exceptional strings among all permutations of $\bfa$,  
and it is of interest to investigate the asymptotic behavior of
$\epsilon(\bfa)$  as the length $n$ of the string $\bfa$ tends to infinity.
In particular, if $\bfa^{(n)}=(1,2,\dots,n)$, does
$\epsilon(\bfa^{(n)})$ converge to $0$ as $n\to\infty$?  Is it the case
that $\epsilon(\bfa)\to0$ \emph{uniformly in $\bfa$} as the length $n$
of the string $\bfa$ tends to infinity? 

A related problem is to characterize \emph{all}
exceptional strings.  The families of strings constructed in the proof
of Theorem \ref{thm:nontrivial-symmetries}(ii) are examples of such
strings, but there exist exceptional strings that are not part of these
families (nor the generalized families mentioned in Remark
\ref{rem:generalized-stable-families}).  For example, one can verify
that any string of the form $(t+1,1,t+3,t+2)$, where $t\in\NN$, has a
nontrivial symmetry given by $(t+2,1,t+1,t+3)$.

Another circle of questions concerns the set of frequencies of 
the strings $\sigma(\bfa)$ as $\sigma$ runs through all
permutations of a given string $\bfa$.  What can one say about the
maximal and minimal frequencies in this set,
and the permutations under which these extremal frequencies are attained?

\vskip20pt\noindent {\bf Acknowledgements.} 
This work originated with an undergraduate
research project carried out at the \emph{Illinois Geometry Lab}
(IGL) at the University of Illinois; we thank the IGL for providing this
opportunity.  We also thank the referee for a thorough reading of this
paper and helpful comments and suggestions.



\begin{thebibliography}{10}\footnotesize

\bibitem{borel1909}
{\'E}.~Borel, Les probabilit{\'e}s d{\'e}nombrables et leurs applications
  arithm{\'e}tiques, \textit{Rend. Circ. Mat. Palermo} \textbf{27} (1909)(1),
  247--271.

\bibitem{neverending-fractions-book}
J.~Borwein, A.~van~der Poorten, J.~Shallit, and W.~Zudilin, \textit{Neverending
  Fractions}, vol.~23 of \textit{Australian Mathematical Society Lecture
  Series}, Cambridge University Press, Cambridge, 2014.

\bibitem{dajani-kraaikamp-book}
K.~Dajani and C.~Kraaikamp, \textit{Ergodic Theory of Numbers}, vol.~29 of
  \textit{Carus Mathematical Monographs}, Mathematical Association of America,
  Washington, DC, 2002.

\bibitem{hardy-wright2008}
G.~H. Hardy and E.~M. Wright, \textit{An Introduction to the Theory of
  Numbers}, Oxford University Press, Oxford, sixth edn., 2008.

\bibitem{iosifescu-kraaikamp-book}
M.~Iosifescu and C.~Kraaikamp, \textit{Metrical Theory of Continued Fractions},
  vol. 547 of \textit{Mathematics and its Applications}, Kluwer Academic
  Publishers, Dordrecht, 2002.

\bibitem{khinchin-book}
A.~Y. Khinchin, \textit{Continued Fractions}, University of Chicago Press,
  Chicago, Ill.-London, 1964.

\bibitem{kuzmin1929}
R.~Kuzmin, Sur un probl{\`e}me de {G}auss, in \textit{Atti del Congresso
  Internazionale dei Matematici, Bolognia}, 1928, pp. 83--90.

\bibitem{levy1929}
P.~L{\'e}vy, Sur les lois de probabilit{\'e} dont d{\'e}pendent les quotients
  complets et incomplets d'une fraction continue, \textit{Bull. Soc. Math.
  France} \textbf{57} (1929), 178--194.

\bibitem{niven1956}
I.~Niven, \textit{Irrational Numbers}, vol. No. 11 of \textit{The Carus
  Mathematical Monographs}, Mathematical Association of America, 1956.

\bibitem{perron-book}
O.~Perron, \textit{Die {L}ehre von den {K}ettenbr\"uchen. {B}d {I}.
  {E}lementare {K}ettenbr\"uche}, B. G. Teubner Verlagsgesellschaft, Stuttgart,
  1954.

\bibitem{vandehey2016}
J.~Vandehey, New normality constructions for continued fraction expansions,
  \textit{J. Number Theory} \textbf{166} (2016), 424--451.

\end{thebibliography}

\end{document}